\documentclass[11pt,twoside]{amsart}
\usepackage{amsfonts}
\usepackage{amsmath}
\usepackage{amssymb}
\usepackage[latin1]{inputenc} 
\usepackage{mathrsfs}

\newtheorem{theo}{Theorem}
\newtheorem{pro}{Proposition}
\newtheorem{cor}{Corollary}
\newtheorem{lem}{Lemma}
\theoremstyle{definition}
\newtheorem{defn}{Definition}
\newtheorem{propr}{Properties}

\theoremstyle{remark}
\newtheorem{rem}{Remark}

\newtheorem{exe}{Example}

\numberwithin{equation}{section}

\newcommand{\cal}       {\mathcal}
\def\finpr{\hfill \hbox{
\vrule height 1.453ex  width 0.093ex  depth 0ex \vrule height
1.5ex  width 1.3ex  depth -1.407ex\kern-0.1ex \vrule height
1.453ex  width 0.093ex  depth 0ex\kern-1.35ex \vrule height
0.093ex  width 1.3ex  depth 0ex}}
\newenvironment{proofof}{{\noindent {\it Proof of }}}{\hfill \finpr \\ }
\def\bC{{\Bbb C}}
\def\bR{{\Bbb R}}

          \def\C{\cal C}          \def\b{\beta}
                      
          \def\L{\Lambda}         \def\l{\lambda}
       \def\ds{\displaystyle}  \let\w=\wedge
           \let\L=\longrightarrow  
         \def \C{\cal C}         \let\l=\rightarrow
               
                 \def\ds{\displaystyle}
\let\ov=\overline
    
\def\1{1\!\rm l}

\title[Complex Hessian Operator, $m$-capacity and $m$-Potential]{Complex Hessian Operator, $m$-capacity, Cegrell's classes and $m$-Potential associated to a
Positive Closed Current}
\author[A. Dhouib and F. Elkhadhra]{Abir Dhouib and Fredj Elkhadhra}
\address{D\'epartement de Math\'ematique\\ Ecole Superieur des Sciences et de Technologie de Hammam Sousse\\Rue Lamine Abassi 4011 Hammam-Sousse Tunisie.}
\email{fredj.elkhadhra@fsm.rnu.tn, abirdhouib2013@hotmail.fr}

\subjclass[2010]{Primary 32C30; Secondary 31A15}
\keywords{Positive current, Capacity}
\begin{document}
\begin{abstract} 
    In this paper we firstly introduce the concepts of capacity and Cegrell's classes associated to any $m$-positive closed current $T$. Next, after investigating the most imporant related properties, we study the definition and the continuity of the complex hessian operator in several cases, generalizing then the work of Demailly and Xing in this direction. We also prove a Xing-type comparison principle for the analogous Cegrell class $\mathcal{F}^{m,T}$ of negative $m$-subharmonic functions. Finally, we generalize the work of Ben Messaoud-El Mir on the complex Monge-Amp\`ere operator and the Lelong-Skoda potential associated to a positive closed current.
\end{abstract} 
\maketitle
\tableofcontents
\section{Introduction}
    Let $\Omega$ be a bounded open subset of $\bC^n$. Denote by ${\cal {PSH}}(\Omega)$ the set of plurisubharmonic (psh) functions on $\Omega$. Denote also by ${\mathscr D}_p^+(\Omega)$ (resp.${\mathscr C}_p^+(\Omega)$) the convex cone of positive $(p,p)$-forms (resp. positive currents  of bidimension $(p,p)$) on $\Omega$. Throughout the paper, $\beta$ is  the standard K\"ahler form on $\bC^n$ and ${\rm Supp}T$ is the support of a given current $T$. Beside the introduction the paper has four sections. In Section 2 we give a short discussion on the notion of $m$-positivity of forms and currents, introduced recently by Lu \cite{[10]}. This notion serves as a generalization and the analogue of the well-known theory of positivity. In \cite{[5]}, Cegrell introduced and studied three importants classes of negative psh functions. Among the fundamental properties of such classes, Cegrell obtained the biggest domain of definition of the complex Monge-Amp\`ere operator. In \cite{[9]}, the authors associated to every closed positive current $T$ an analogous classes. In particular, they generalize some properties obtained by Cegrell for the trivial current $T=1$. Later on, building the existence of a local solution of the complex hessian equation $(dd^c.)^m\wedge\beta^{n-m}=0$, Lu \cite{[10]} extends the work of Cegrell to the context of hessian complex theory. Namely, he consider the Cegrell's classes of $m$-subharmonic ($m$-sh) functions relatively to the strongly positive current $T=\beta^{n-m}$. In Section 3, we associate to each closed $m$-positive current $T$, the notion of capacity. In this study and similarly as in \cite{[7]}, we prove that every $m$-sh bounded function is continuous away from an open set of arbitrarily small capacity and therefore we obtain a Xing-type comparison principle inequality for $m$-sh functions. We establish also that the complex hessian operator $T\wedge(dd^c.)^p$ converges for monotonic limits of $m$-sh functions which are bounded only near $\partial\Omega\cap{\rm Supp}T$. This is essentially the work of Demailly \cite{[8]} in the border case $m=n$. Next, using the quasicontinuity of $m$-sh bounded function we prove that the monotonicity condition can be relaxed to a convergence in the sense of capacity, but under certain hypothesis on the relative p\^ole sets. Similarly as in \cite{[9]} and \cite{[10]}, in Section 4, we associate to each $m$-positive closed current, the analogous pluricomplex energy classes of Cegrell. Some properties developed in \cite{[9]} and \cite{[10]}, are then generalized. Namely, we show that the complex hessian operator previously studied in section 3, is well defined for the Cegrell's classes. Moreover, in connection with these classes, we prove a Xing-type comparision principle inequality, which generalizes the one proved by \cite{[11]} for the trivial current $T=1$. The purpose of Section 5, is to extend the main work of Ben Messaoud-El Mir \cite{[3]} on the Monge-Amp\`ere operator and on the local potential relatively to a positive closed current, to the complex hessian theory. To this aim, we replace the well-known Newton kernel used by \cite{[3]}, by an $(n-m+1)$-sh function of Riesz-type kernel.
\section{$m$-positivity of forms and currents}
 According to \cite{[10]}, a real $(1,1)$-form $\alpha$ is said $m$-positive on $\Omega$ if at every point of $\Omega$ we have
$\alpha^j\wedge\beta^{n-j}\geq 0,\ \forall j=1,...,m.$ By duality a current $T$ of bidimension $(p,p)$ on $\Omega$, $p\leq m$, is said $m$-positive if
$T\wedge\alpha_1\wedge...\wedge\alpha_p \geq 0,$ for all $m$-positive $(1,1)$-forms $\alpha_1,...,\alpha_p$.
\begin{rem} It is not hard to see that the notion of $m$-positivity of $(1,1)$-forms coincides with the standard one when $m=n$. This is not the case if $m<n$:  in fact it is clear that the form $\alpha=idz_1\wedge d\ov z_1+idz_2\wedge d\ov z_2-\frac{i}{2}dz_3\wedge d\ov z_3$ in $\bC^3$, is $2$-positive but not positive.
\end{rem}
The following lemma will be essential for our work
\begin{lem} (See \cite{[4]}) Let $1\leq p\leq m$. If $\alpha,...\alpha_p$ are $m$-positive $(1,1)$-forms then $\alpha_1\wedge...\wedge\alpha_p\wedge\b^{n-m}\geq 0$.
\end{lem}
 Our aim now is to formulate a more general definition of $m$-positivity that has the requirement of being compatible with the concept of the standard notion of positivity :
\begin{defn} Let $\varphi$ be a real $(p,p)$-form on $\Omega$ and $T$ be a current of bidimension $(p,p)$ on $\Omega$. Let $p\leq m\leq n$, then we say that
\begin{enumerate}
\item $\varphi$ is $m$-positive on $\Omega$ if at every point of $\Omega$ we have
$$\varphi\wedge\beta^{n-m}\wedge\alpha_1\wedge...\wedge\alpha_{m-p} \geq 0,\qquad\forall \alpha_1,...,\alpha_{m-p}\ \ m{\rm -positive\ forms}.$$
\item $\varphi$ is $m$-strongly positive on $\Omega$ if
$$\varphi=\ds\sum_{k=1}^N\lambda_k\alpha_1^k\wedge...\wedge\alpha_{p}^k,$$
where $\alpha_1^k,...,\alpha_{p}^k$, are $m$-positive forms on $\Omega$ and $\lambda_k\geq 0$.
\item $T$ is $m$-positive if $\langle T,\varphi\rangle\geq 0,$ $\forall \varphi$ $m$-strongly positive $(p,p)$-form on $\Omega$.
\item $T$ is $m$-strongly positive if $\langle T,\varphi\rangle\geq 0,$ $\forall \varphi$ $m$-positive $(p,p)$-form on $\Omega$.
\end{enumerate}
\end{defn}
\begin{rem} Notice that when $m=n$, we recover the well-known notions of positivity and strongly positivity. Moreover, it is clear that $m$-strongly positivity implies $m$-positivity and the notion of $m$-positivity of currents coincides with the one given by Lu \cite{[10]}. Let also stress that if $p=1$, then our definition 1 for $m$-positivity of forms is equivalent with the notion of $m$-positive forms given by \cite{[10]}. This means that if $\alpha$ is a real $(1,1)$-form, then $\alpha\wedge\beta^{n-m}\wedge\alpha_1\wedge...\wedge\alpha_{m-1} \geq 0,\ \forall \alpha_1,...,\alpha_{m-1}$ $m$-positive forms is equivalently saying that $\alpha^j\wedge\beta^{n-j}\geq 0, \forall j=1,...,m$. Indeed, we can use lemma 1 combined with the preliminary part of \cite{[12]}. In particular, as in the border case $m=n$, if $p=1$, there is no difference between $m$-positivity and $m$-strongly positivity.
\end{rem}
\begin{exe} Notice that unlike the complex case ($m=n$), starting from an $m$-positive form, we cannot define an $m$-positive current. However, if $\varphi$ is an $m$-positive (resp.$m$-strongly positive) $(p,p)$-form on $\Omega$ then, the form $\varphi\wedge\beta^{n-m}$ define an $m$-positive (resp.$m$-strongly positive) current of bidimension $(m-p,m-p)$ on $\Omega$. More generally, if $X$ is a pure $p$-dimensional analytic subset of $\Omega$, such that $n\leq m+p$, then, in view of definition 1, we see that $[X]\wedge\beta^{n-m}$ is an $m$-strongly positive closed current of bidimension $(p+m-n,p+m-n)$ and supported by $X$.
\end{exe}
For convenience, we will denote by ${\mathscr D}_p^m(\Omega)$ (resp.${\mathscr C}_p^m(\Omega)$) the convex cone of $m$-positive $(p,p)$-forms (resp. $m$-positive currents of bidimension $(p,p)$) on $\Omega$.
\begin{defn}
A function $u:\Omega\L\bR\cup\{-\infty\}$ is called $m$-subharmonic if it is subharmonic and
$$dd^cu\wedge\alpha_1\wedge...\wedge\alpha_{m-1}\wedge\b^{n-m}\geq 0,$$
for all $m$-positive $(1,1)$-forms $\alpha_1,...,\alpha_{m-1}$. Denote by ${\mathscr P}_m(\Omega)$ the class of $m$-sh functions on $\Omega$.
\end{defn}
 As a direct consequence, it is clear that $dd^cu\wedge\beta^{n-m}$ is an $m$-positive current for each $u\in{\mathscr P}_m(\Omega)$ and $1\leq m\leq n$. The four first assertions in the following proposition was presented in many papers (see \cite{[10]} or \cite{[12]}).
\begin{pro} \
\begin{enumerate}
\item If $u$ is of class ${\mathscr C}^2$ then $u\in{\mathscr P}_m(\Omega)$ if and only if $dd^cu$ is $m$-positive on $\Omega$.
\item If $u\in{\mathscr P}_m(\Omega)$, then the standard regularization $u_j=u\star\chi_j\in{\mathscr P}_m(\Omega_j)\cap{\mathscr C}^\infty(\Omega_j)$, where $\Omega_j=\{x\in\Omega:\ d(x,\partial\Omega)>1/j\}$. Moreover, $(u_j)_j$ decreases pointwise to $u$.
\item If $(u_\alpha)_\alpha$ is a family of $m$-sh functions, $u=\sup_\alpha u_\alpha<+\infty$ and $u$ is upper semicontinuous then $u$ in $m$-sh.
\item ${\cal {PSH}}(\Omega)={\mathscr P}_n(\Omega)\subset{\mathscr P}_{n-1}(\Omega)\subset\cdots\subset{\mathscr P}_1(\Omega)={\cal {SH}}(\Omega):=\{u,\ {\rm subharmonic\ on}\ \Omega\}$.
\item ${\mathscr D}_p^+(\Omega)={\mathscr D}_p^n(\Omega)\subset{\mathscr D}_p^{n-1}(\Omega)\subset\cdots\subset{\mathscr D}_p^p(\Omega)$.
\item ${\mathscr C}_p^p(\Omega)\subset{\mathscr C}_p^{p+1}(\Omega)\subset\cdots\subset{\mathscr C}_p^n(\Omega)={\mathscr C}_p^+(\Omega)$.
\end{enumerate}
\end{pro}
As an immediate consequence of statement $(1)$, we see that if $u_1,...,u_p\in{\mathscr C}^2(\Omega)\cap{\mathscr P}_m(\Omega)$, $p\leq m$, then the $(p,p)$-form $dd^cu_1\wedge\cdots\wedge dd^cu_p$ is $m$-strongly positive on $\Omega$. On the other hand, it is quite easy to establish that the three later inclusions in Prop.1 are strict in general. In fact, thanks to the first statement, we see that $u=|z_1|^2+|z_2|^2-\frac{1}{2}|z_3|^2\in{\mathscr P}_2(\bC^3)\smallsetminus {\cal {PSH}}(\bC^3)$. For the assertion $(5)$, let $\varepsilon>0$, $\gamma=idz_1\wedge d\ov z_1+idz_2\wedge d\ov z_2+idz_3\wedge d\ov z_3-\varepsilon idz_4\wedge d\ov z_4$ and $\varphi=-\frac{i}{2}dz_1\wedge d\ov z_1\wedge idz_2\wedge d\ov z_2+idz_3\wedge d\ov z_3\wedge idz_4\wedge d\ov z_4$. A direct computation shows that $\gamma$ is $3$-positive, $\varphi\wedge\beta\wedge\gamma\leq 0$ and $\varphi\wedge\beta^2\geq 0$. This means that $\varphi\in{\mathscr D}_2^2(\bC^4)\smallsetminus{\mathscr D}_2^3(\bC^4)$. Concerning the cone of $m$-positive currents, let $\alpha$ be the $(1,1)$-form used in remark 1 and denote by $T_1=[z_1=0]$ the current of integration on $\{z_1=0\}$ in $\bC^3$. It is not hard to see that $T_1\wedge\alpha^2$ is negative. It follows that $T_1\in{\mathscr C}_1^+(\bC^3)\smallsetminus{\mathscr C}_1^2(\bC^3)$. For the intermediate cones, let us consider
$$T_2=\beta\wedge\alpha=2idz_1\wedge d\ov z_1\wedge idz_2\wedge d\ov z_2+\frac{i}{2}dz_1\wedge d\ov z_1\wedge idz_3\wedge d\ov z_3+\frac{i}{2}dz_2\wedge d\ov z_2\wedge idz_3\wedge d\ov z_3 .$$
Thanks to lemma 1, we have $T_2\in{\mathscr C}_1^2(\bC^3)$. Since $\gamma=idz_1\wedge d\ov z_1+idz_2\wedge d\ov z_2-idz_3\wedge d\ov z_3$ is $1$-positive and $T_2\wedge\gamma$ is negative, we conclude that $T_2\not\in{\mathscr C}_1^1(\bC^3)$.
\begin{rem} It should be noted that for $m<n$, the above classes ${\mathscr P}_m(\Omega), {\mathscr D}_p^m(\Omega)$ and ${\mathscr C}_p^m(\Omega)$ are not preserved under direct or inverse images by holomorphic maps. For example, let $\pi$ be the projection defined by $\pi(z,t)=z$, $z\in\bC^3, t\in\bC$. Let $T=\beta^2$ in $\bC^3$, then by lemma 1, $T$ is $2$-positive. However, $\pi^\star T$ is not $2$-positive since $\pi^\star T\wedge dd^c|z|^2\wedge\alpha\leq 0$, where $\alpha=dd^c|z|^2-\varepsilon idt\wedge \overline dt$ is a $2$-positive form in $\bC^4$, for $0<\varepsilon<1$. It is not difficult to formulate similar examples for the other cones as well as for the direct images.
\end{rem}
In order to more understand the cone of $m$-positive currents, we shall prove :
\begin{pro} Let $T\in{\mathscr C}_p^p(\bC^n)$ be closed, then $T$ is generated by $\beta^{n-p}$, i.e, there exists a constant $c\geq 0$ such that $T=c\beta^{n-p}$.
\end{pro}
\begin{proof} Assume that $p=1$, then $T$ can be written as $T=\sum_{1\leq j,k\leq n}T_{jk}i{\ds\mathop {dz_j\wedge d\ov z_k}^\vee}$, where $\vee$ means that we omit the form $dz_j\wedge d\ov z_k$ in the saturate one. Let us consider the real $(1,1)$-forms
$$\alpha_k=idz_1\wedge d\ov z_1+...-(n-1)idz_k\wedge d\ov z_k+...+idz_n\wedge d\ov z_n .$$
It is clear that for every $k\in\{1,...,n\}$, $\alpha_k$ is $1$-positive. Since $T$ is $1$-positive, one have $T\wedge\alpha_k\geq 0,$ $k=1,...,n$ in the sense of currents. Therefore, we obtain a system of $n$-inequalities formed by
$$-(n-1)T_{kk}+\sum_{s\not=k}T_{ss}\geq 0.$$
It is not hard to see that these inequalities are in fact equalities. Hence,  $(n-1)T_{kk}=\sum_{s\not=k}T_{ss}$ for $k=1,...,n$. By a simple computation we show that $T_{jj}=T_{kk}$ for $j\not=k$. Similarly we can check the case $p=n-1$. To prove the other cases $1<p<n-1$, we argue by induction on $n$. The result is clear when $n=2$ and assume it for $n-1$. Let $T$ be a current as in the proposition 2. According to \cite{[3]}, the slice $T_{|L}$ exists for $L$ in a non pluripolar subset of the grassmanien $G(n-1,n)$. We claim that $T_{|L}\in{\mathscr C}_{p-1}^{p-1}(L)$. In fact, let $\alpha_1,...,\alpha_{p-1}\in{\mathscr D}_1^{p-1}(L)$. Without loss of generality, assume that $L=\{z_n=0\}$ and setting $\gamma_j(z',z_n)=\alpha_j(z')+tdd^c|z_n|^2$, $t>0$, $j=1,...,p-1$. For $t$ sufficiently large, we have $\gamma_j\in{\mathscr D}_1^p(\bC^n)$. Indeed, let $'\lambda=(\lambda_1,...,\lambda_{n-1})$ be the eigenvalues of $\alpha_1$, then $\lambda=(\lambda_1,...,\lambda_{n-1},t)$ are the eigenvalues of $\gamma_1$. By \cite{[1]}, we have $\frac{\partial}{\partial t}H_p(\lambda)=H_{p-1}('\lambda)$, where $H_m(\mu)$ is the symmetric functions of order $m$ of the vector $\mu$. Since $\alpha_1$ is $(p-1)$-positive, $H_{p-1}('\lambda)>0$. It follows that $H_p(\lambda)>0$ for $t$ sufficiently large. Therefore,
$$T_{|L}\wedge\alpha_1\wedge...\wedge\alpha_{p-1}=T\wedge dd^c\log|z_n|\wedge\gamma_1\wedge...\wedge\gamma_{p-1}\geq 0.$$ By the induction hypothesis, it follows that $T_{|L}=c_L(dd^c|z'|^2)^{n-p}$.
\end{proof}
\section{$m$-capacity and continuity of the complex hessian operator}
The purpose of this section is towfold. First, similarly as in \cite{[7]} and \cite{[10]}, we introduce the notion of capacity associated to an $m$-positive closed current $T$. At the same times, we discuss some results and properties of such capacity, and we point out that some of them are not automatically repeated as in the trivial current $T=1$, or $T=\beta^{n-m}$. Second, we study the continuity of the complex hessian  operator for decreasing sequences of $m$-sh functions bounded near the boundary, as well as for sequences of $m$-sh functions converging in the sense of capacity.
\subsection{Relative $m$-capacity associated to an $m$-positive closed current}
According to \cite{[10]}, the complex hessian  operator $dd^cu_1\wedge\cdots\wedge dd^cu_k\wedge T$ is well defined for an $m$-positive closed current $T$ and $m$-sh locally bounded functions $u_1,...,u_k$. Moreover, acting on locally bounded $m$-sh decreasing sequence, such operator is continuous. As a consequence and similarly to \cite{[7]}, we associate to each $m$-positive closed current the following relative $m$-capacity :
\begin{defn}
Let $\Omega$ be an open set of $\mathbb{C}^n$, $K\subset\Omega$ a compact and $T$ an $m$-positive closed current of bidimension $(p,p)$ on $\Omega$, $m\geq p\geq 1$. We define the $m$-capacity of $K$ relatively to $T$ by:
$$cap_{m,T}(K,\Omega)=cap_{m,T}(K):=\sup \left\{ \int_{K} (dd^c u)^{p}\wedge T,~ u\in {\mathscr P}_m(\Omega), ~0\leq v\leq 1\right\},$$
and for every $E\subset\Omega$, $cap_{m,T}(E,\Omega)=\sup\left\{cap_{m,T}(K),\ K\ {\rm compact\ of}\ \Omega\right\}$.
\end{defn}
This capacity generalizes the one given in \cite{[10]} for the strong positive current $T =\beta^{n-m }$ and the one in \cite{[7]} for the case $m = n$, i.e, $T$ is a closed positive current. Such capacity shares the sames properties as the preceding capacities. Namely, we have
\begin{propr} \
\begin{enumerate}
\item If $E$ is Borel set, then $cap_{m, T}(E,\Omega)= \sup \left\{\int_E (dd^c v)^p\wedge T,\quad v \in  {\mathscr P}_m(\Omega) ,0\leq v\leq 1 \right\}; $
\item If $E_1 \subset E_2$ then  $cap_{m, T}(E_1,\Omega)\leq  cap_{m, T}(E_2,\Omega);$
\item If $E \subset \Omega_1\subset\Omega_2$ then  $cap_{m, T}(E,\Omega_1)\geq  cap_{m, T}(E,\Omega_2);$
\item If $E_1,E_2,\cdots $ are Borel sets of $\Omega$, then
$cap_{m, T}\left(\cup_{j\geq 1} E_j, \Omega \right)\leq \sum_{j=1}^{+\infty} cap_{m, T}(E_j,\Omega).$
\item If $E_1 \subset E_2 \subset \cdots$ are Borel sets of $\Omega$, we have :
$cap_{m, T}\left(\cup_{j\geq 1} E_j, \Omega \right)= \ds\lim_{j\rightarrow +\infty}cap_{m, T}(E_j,\Omega).$
\end{enumerate}
\end{propr}
In this direction, we state the following definition :
\begin{defn} A subset $A\subset \Omega$ is said $(m,T)$-pluripolar in $\Omega$ if $cap_{m,T} (A,\Omega)=0.$ One say that $A$ is locally $(m,T)$-pluripolar, if for avery $a\in A$, there exists an open neighborhood $V$ of $a$ such that $A\cap V$ is $(m,T)$-pluripolar in $V$ i.e, $cap_{m,T} (A\cap V,V)=0.$
\end{defn}
\begin{rem} \
\begin{enumerate}
\item Following the terminology of \cite{[10]}, a Borel set $A$ is $m$-polar (i.e, $(m,\beta^{n-m })$-pluripolar in the sense of definition 4) in $\Omega $ if and only if $A\subset\{z\in\Omega,\ v(z)=-\infty\}$, where $v\in{\mathscr P}_m(\Omega)$.
\item Assume that $0\in\Omega$ and let $L$ be a complex linear space of dimension $p$ in $\bC^n$. By a unitary change of coordinates, we assume that $L=\bC^p\times\{0\}$. Select an integer $m$ such that $p+m\geq n$. Let ${\mathscr O}$ be an open subset of $\Omega$, $u\in{\mathscr P}_m(\Omega,[0,1]),$ and $i:L\cap\Omega\hookrightarrow \Omega$ is the inclusion map. Then, we see that
$$\ds\int_{{\mathscr O}}
[L]\wedge\beta^{n-m}\wedge(dd^c u)^{p+m-n}=\ds\int_{{\mathscr O}\cap L}(i^{\star}\beta)^{p-(m+p-n)}\wedge(dd^c (i^{\star}u))^{p+m-n}.$$
Thanks to a result of \cite{[1]}, $u_{|{L\cap\Omega}}=i^{\star}u$ is $(m+p-n)$-sh. Therefore, if we consider the current $T=[L]\wedge\beta^{n-m}$, by the above equality, we deduce the equivalence: ${\mathscr O}$ is locally $(m,T)$-pluripolar if and only if $L\cap {\mathscr O}$ is locally $(m+p-n)$-polar in $L\cap\Omega$.
\item Assume that $T$ is an $m$-positive closed current of bidimension $(p,p)$ on $\Omega$ and let $E$ be a Borel subset of $\Omega$. From Prop.1 we have: $cap_{m,T}(E)\geq cap_{m+1,T}(E)\geq\cdots\geq cap_{n,T}(E)=C_T(E)$ (with the notation of \cite{[7]}). In particular, if $E$ is $(m,T)$-pluripolar then $E$ is $T$-pluripolar in the sense of \cite{[7]}. Before ending this discussion, it should be noted that the equivalence of Prop.2.3 in \cite{[9]} is far from being true. In fact if $A$ is $T$-pluripolar then $A$ is $T$-negligible (the trace measure of $A$ is zero) but the converse is false. Indeed, let $\Omega$ be the unit open ball in $\bC^2$, $T=dd^c\log|z_1|$, and let $A$ be a compact subset of $\Omega\cap\{z_1=0\}$, such that $A$ is Lebesgue-negligible but not polar.
\end{enumerate}
\end{rem}
By repeating the arguments of \cite{[9]}, we can prove :
\begin{pro} If $A\subset\Omega$ is locally $(m,T)$-pluripolar then $A$ is $(m,T)$-pluripolar.
\end{pro}
One of the most important properties of locally bounded $m$-sh functions is quasicontinuity with respect to  $cap_{m,T}$ : every $m$-sh locally
bounded function is continuous outside an open set with arbitrarily small capacity $cap_{m,T}$. By adaptation of the proof of \cite{[7]}, we obtain:
\begin{theo}
Let $\Omega$ be a bounded open subset of $\mathbb{C}^{n}, u\in {\mathscr P}_m(\Omega)\cap L_{loc}^\infty(\Omega)$ and $T$ is an $m$-positive closed current of bidimension $(p,p)$ on $\Omega$. Then for every $\varepsilon >0$, there exists an open set ${\mathscr O}$ of $\Omega$ such that $cap_{m,T}({\mathscr O},\Omega) <\varepsilon$ and $u$ is continuous on $\Omega\smallsetminus {\mathscr O}$.
\end{theo}
\begin{rem} \
\begin{enumerate}
\item In \cite{[10]} Lu, was established Thm.1 for the particular case $T=\beta^{ n-m }$ but $u$ is only $m$-sh. In our situation the required assumption {\it $u$ locally bounded} is essential as shown the following example:
$$\Omega=\Delta^3:={\rm the\ polydisc\ of}\ \mathbb{C}^3,\ \ T = [z_1=0]\wedge\beta,\ \ u(z_1,z_2)=\log |z_1|.$$
It is clear that $T$ is a $2$-positive closed current. Moreover, $u$ is a $2$-sh function and is discontinous on the support of $T$, which has a strictly positive capacity.
\item Let $u$ be an $m$-sh locally bounded function in $\Omega$. Assume that $X$ is an analytic subset of $\Omega$ of dimension $p$, with $p+m\geq n$. It is important to point out that the Theorem of quasicontinuity of \cite{[10]} don't gives any information about the regularity of $u$ near $X$. By applying Thm.1 for $T=[X]\wedge\beta^{n-m}$, we see that $u$ is continuous on $X$ minus a subset of arbitrarily small volume in $X$.
\end{enumerate}
\end{rem}
Before terminating this discussion, we state some interesting related problems pertaining to the notion of the capacity $cap_{m,T}$:
\vskip0.1cm
{\bf P1.} Assume that ${(K_j)}_j$ is a decreasing sequence of compact subsets of $\Omega$, then it is clear that

\hskip0.9cm $cap_{m,T}(K_{j})$ decreases. Is that $cap_{m,T} (\cap
K_{j})=\displaystyle\lim_{j\l +\infty}cap_{m,T} (K_{j})?$
\vskip0.1cm
{\bf P2.} Let $K$ be a compact subset of $\Omega$. Is there a function $u\in{\mathscr P}_m\left(\Omega,[0,1]\right)$, such that

\hskip0.9cm $cap_{m,T} (K)=\int_K(dd^c u)^p\wedge T ?$
\vskip0.1cm
{\bf P3:} Can we caracterize the $(m,T)$-pluripolar subsets of $\Omega$?
\vskip0.2cm
It should be noted here that the above problems was resolved by \cite{[2]} for $m=n, T=1$ and by \cite{[1]} and \cite{[10]} for $T=\beta^{n-m}$.
\subsection{Continuity of the complex hessian operator}
Building on the work of Demailly \cite{[8]}, we show that the complex hessian  operator is well defined for $m$-sh functions which are bounded only near $\partial\Omega\cap{\rm Supp}T$. More precisely, we prove:
\begin{theo} Let $\Omega$ be a bounded strictly pseudoconvex open subset of ${\mathbb{C}}^n$. Assume that $u_1, \cdots, u_k$ are $m$-sh functions on $\Omega$ such that each $u_j$ is bounded near $\partial\Omega\cap{\rm Supp}T$, where $T$ is an $m$-positive closed current of bidimension $(p,p)$ in $\Omega$. Then, by induction the following operator
$$dd^cu_1\wedge\cdots\wedge dd^cu_k\wedge T=dd^c(u_1dd^cu_2\wedge\cdots\wedge dd^cu_k\wedge T),$$
is well defined in $\Omega$. Moreover, assume that $u_1^j,\cdots, u_k^j$ are decreasing sequences of $m$-sh functions converging pointwise to $u_1,\cdots,u_k$, then we have
\begin{enumerate}
\item $u_1^j dd^cu_2^j\wedge\cdots\wedge dd^cu_k^j\wedge T$ converges weakly to $ u_1
dd^cu_2\wedge\cdots\wedge dd^cu_k\wedge T$.
\item $dd^cu_1^j\wedge\cdots\wedge dd^cu_k^j\wedge T$ converges weakly to $dd^cu_1\wedge\cdots\wedge dd^cu_k\wedge T$.
\end{enumerate}
\end{theo}
\begin{proof} We argue as in \cite{[8]}, then without loss of generality we may assume that $\Omega=\{\rho<0\}$, where $\rho$ is a smooth and strictly psh function on $\overline\Omega$. If $k=1$, by using the function $\max(u_1,s)$, we proceed as in \cite{[8]} to prove that $u_1T$ has locally bounded mass in $\Omega$. By hypothesis, there exist an open set ${\mathscr U}_T$ containing ${\rm Supp}T\cap\Omega$ (shrinking $\Omega$ if necessary) and an open set ${\mathscr U}_u$ containing all polar sets $u_j^{-1}(-\infty)\cap\Omega$ such that $\overline{\mathscr U}_T\cap\overline{\mathscr U}_u\Subset\Omega$ and $u_j$ are bounded on ${\mathscr U}_T\smallsetminus{\mathscr U}_u$. Then by using induction, one can define the current
$$dd^cu_1\wedge\cdots\wedge dd^cu_k\wedge T=dd^c(u_1dd^cu_2\wedge\cdots\wedge dd^cu_k\wedge T).$$
 Now, let $\delta>0$ sufficiently small such that $\overline{\mathscr U}_T\cap\overline{\mathscr U}_u\subset\Omega_\delta$, here $\Omega_\delta=\{\rho<-\delta\}$. In order to complete the proof, we shall repeat the arguments of Lu \cite{[10]} on ${\mathscr U}_T\smallsetminus\Omega_\delta$.
\end{proof}
\begin{rem}
It is interesting to note that if $u_j$ are locally bounded, then we recover a result of \cite{[10]}. As a consequence, we mension that the wedge product $\gamma_1\wedge...\wedge\gamma_k\wedge T,$ is well defined for $\gamma_j=dd^cu_j$ or $\gamma_j=dv_j\wedge d^cw_j+dw_j\wedge d^cv_j$, where $u_j,v_j$ et $w_j$ are $m$-sh locally bounded functions. In fact, we just see that
$$2(dv_j\wedge d^cw_j+dw_j\wedge d^cv_j)=dd^c(v_j+w_j)^2-dd^cv_j^2-dd^cw_j^2-v_jdd^cw_j-w_jdd^cv_j .$$
\end{rem}
Notice that the notion of convergence in capacity become a good tools for the convergence of the complex hessian  operator. Recall that a sequence of functions $(u_j)_j$ defined on $\Omega$ is said converges to a function $u$ in the capacity $cap_{m, T}$ on $E$ if for every $\delta>0$, we have :
$$\ds\lim_{j\rightarrow +\infty} cap_{m, T}( E\cap \left\{|u-u_j|> \delta\right\})=0.$$
\hskip0.3cm The purpose now is to define the complex hessian  operator $(dd^c u)^{p}\wedge T$, for a given $m$-sh function $u$ which is the limit in the sense of $cap_{m,T}$-capacity of $u_j$, where $u_j$ are $m$-sh functions uniformly bounded near $\partial\Omega\cap{\rm Supp}T$ and such that the measures $(dd^c u_j)^{p}\wedge T$ have small mass on any set of small $cap_{m,T}$-capacity. Recall that a sequence of positive measures $(\mu_j)_j$ is said to be {\it uniformly absolutely continuous} with respect to $cap_{m,T}$ in a set $E$ (i.e, $\mu_j\ll cap_{m,T}$), if for each $\varepsilon>0$ there exists a constant $\delta>0$ such that for each Borel subsets $F\subset E$ with $cap_{m,T}(F)<\delta$ the inequality $\mu_j(F)<\varepsilon$ holds for all $j$. Therefore, we prove :
\begin{theo} Let $\Omega$ be a bounded strictly pseudoconvex open subset of ${\mathbb{C}}^n$, $T$ an $m$-positive closed current of bidimension $(p,p)$ on $\Omega$ and $u\in{\mathscr P}_m(\Omega)$ such that $u$ is bounded near $\partial\Omega\cap{\rm Supp}T$. Assume that
\begin{enumerate}
\item There exists a sequence $(u_{j})_j$ of bounded $m$-sh functions in $\Omega$ such that $u_j$ are uniformly bounded near $\partial\Omega\cap{\rm Supp}T$ for all $j$ and $u_j\l u$ in $cap_{m,T}$ on each $E\Subset\Omega$.
\item  $(dd^c u_j)^{p}\wedge T\ll cap_{m,T}$ uniformly on each subset $E\Subset\Omega$.
\item $\ds\lim_{k\l +\infty}cap_{m,T}(\{u<-k\})=0$.
\end{enumerate}
Then, $(dd^c u_{j})^{p}\wedge T$ converges weakly to $(dd^c u)^{p}\wedge T$ in $\Omega$ and $(dd^c u)^{p}\wedge T\ll cap_{m,T}$ on each subset $E\Subset\Omega$.
\end{theo}
In particular when $m=n$, $u_j$ are uniformly bounded in $\Omega$ and $u$ is locally bounded in $\Omega$, we recover a result of \cite{[7]}. Also, Thm.3 generalizes a theorem of \cite{[13]} for the trivial current $T=1$ and $m=n$. For the proof we need two intermediate lemmas. The first is a Xing's type comparison principle for $m$-positive currents which extends the one given by \cite{[9]} for the border case $m=n$.
 \begin{lem}
Let $\Omega$ be an open bounded subset of $\mathbb{C}^{n}$, $T$ is an $m$-positive closed current of bidimension $(p,p),\ p\geq 1$ on $\Omega$, and let $u,v\in{\mathscr P}_m(\Omega)\cap L^\infty(\Omega)$. Assume that for each $w\in\partial\Omega\cap{\rm Supp}T$, $\ds\liminf_{\xi\rightarrow w}\left(u(\xi) -v(\xi)\right)\geq 0,$ then for any constant $r \geq 1$ and all $w_j\in{\mathscr P}_m(\Omega)$, with $0\leq w_j\leq 1$, $j=1,...,p$, we have
$$\frac{1}{p!}\displaystyle\int_{\left\{u<v\right\}}(v-u)^pdd^cw_1\wedge\cdots\wedge dd^cw_p\wedge T +\displaystyle\int_{\left\{u<v\right\}}(r-w_1)(dd^cv)^{p}\wedge T \leq \displaystyle\int_{\left\{u<v\right\}}(r-w_1)(dd^c u)^{p}\wedge T.$$
\end{lem}
\begin{proof} Using Thm.2 in the case of bounded $m$-sh functions combined with the same arguments used by \cite{[9]}, we can easily obtain the required inequality.
\end{proof}
The following lemma generalizes the well-known result of Bedford-Taylor for the particular case $m=n$ and $T=1$, as well as the result of Lu \cite{[10]} for any $m$ and the current $T=\beta^{n-m}$.
\begin{lem} Let $T$ be an $m$-positive closed current of bidimension $(p,p)$ in an open subset $\Omega$ of $\bC^n$ and let $u,v\in{\mathscr P}_m(\Omega)\cap L_{loc}^\infty(\Omega)$. Then, we have
$${\1}_{\{u>v\}} (dd^c(\max(u,v)))^p\wedge T = {\1}_{\{u>v\}}(dd^c u)^p\wedge T .$$
\end{lem}
\begin{proof} The equality is obvious if $u$ is continuous. For the general case, consider a regularization sequence $(u_j)_j$ that decreases to $u$. Then, ${\1}_{\{u_j>v\}} (dd^c(\max(u_j,v)))^p\wedge T = {\1}_{\{u_j>v\}} (dd^c u_j)^p\wedge T$. It is clear that $f_j=\max(u_j-v,0)\downarrow f=\max(u-v,0)$, $f_j,f$ are locally bounded and quasicontinuous in $\Omega$ (see Thm.1). Therefore, using subaddivity of $cap_{m,T}$, for $\varepsilon>0$, there exists an open set ${\mathscr O}$ of $\Omega$ such that $f_j,f$ are continuous on $\Omega\smallsetminus{\mathscr O}$ and $cap_{m,T}({\mathscr O})<\varepsilon$. By the above equality, we see that $f_j (dd^c(\max(u_j,v)))^p\wedge T = f_j(dd^c u_j)^p\wedge T$. Now we claim that $f_j(dd^c u_j)^p\wedge T$ converges weakly to $f(dd^c u)^p\wedge T$. Indeed, let $\varphi$ be a test function and let $K={\rm support}(\varphi)$ such that $\|\varphi\|_\infty(K)\leq A$. Then,
\begin{equation}
\begin{array}{lcl}
\left|\ds\int_\Omega \varphi[f_j(dd^c u_j)^p\wedge T- f(dd^c u)^p\wedge T]\right|&\leq& A\ds\int_K|f_j-f|(dd^c u_j)^p\wedge T\\& &+\left|\ds\int_\Omega f\varphi[(dd^c u_j)^p\wedge T-(dd^c u)^p\wedge T]\right|.
\end{array}
\end{equation}
Observe that
$$\int_K|f_j-f|(dd^c u_j)^p\wedge T\leq\|f_j-f\|_\infty(K\smallsetminus{\mathscr O})\int_K(dd^c u_j)^p\wedge T+\|f_j-f\|_\infty(K)\int_{K\cap{\mathscr O}}(dd^c u_j)^p\wedge T.$$
The first term in the right hand side inequality converges to $0$ because $f_j$ converge uniformly to $f$ on $K\smallsetminus{\mathscr O}$, while the second term is bounded from above by $A_1\|f_j-f\|_\infty(K)\varepsilon$, where $A_1$ is a constant not depending of $j$. In order to estimate the second integral in $(3.1)$, we can write $f=g+h$, where $g,h$ are two bounded functions such that $g$ is continuous in $\Omega$ and $h=0$ on $\Omega\smallsetminus{\mathscr O}$. Hence, if we set $\mu_j=T\wedge (dd^c u_j)^p-T\wedge (dd^c u)^p$, we obtain
$$\left|\ds\int_\Omega f\varphi\mu_j\right|=\left|\ds\int_\Omega f\varphi[(dd^c u_j)^p\wedge T-(dd^c u)^p\wedge T]\right|\leq\left|\ds\int_\Omega g\varphi\mu_j\right|+\left|\ds\int_\Omega h\varphi\mu_j\right|.$$
Since $\mu_j$ converges weakly to zero and $g$ is continuous, the first integral converges to zero. Also, it is not hard to see that the second integral is bounded from above by $A_2\varepsilon$, $A_2$ is independent of $j$. This completes the proof of the claim, and therefore we obtain the following equality $f (dd^c(\max(u,v)))^p\wedge T=f (dd^cu)^p\wedge T$. By the same arguments, for every $\delta>0$ we have
$$\frac{f}{f+\delta} (dd^c(\max(u,v)))^p\wedge T=\frac{f}{f+\delta} (dd^cu)^p\wedge T.$$
Since $\ds\frac{f}{f+\delta}\uparrow {\1}_{\{u>v\}}$, by letting $\delta$ to $0$, we obtain the lemma.
\end{proof}
\begin{proofof}{\it Theorem3.} By hypothesis $u$ is bounded near $\partial\Omega\cap{\rm Supp}T$, then by Thm.2, $(dd^c u)^{p}\wedge T$ is a positive Borel measure. The proof was divided in two steps.\\
{\bf Step.1} Assume that $u_j$ are uniformly bounded in $\Omega$ and $u$ is locally bounded in $\Omega$. Then, we repeat the proof of \cite{[7]} by using Thm.1 instead of the theorem of quasi-continuity for bounded psh functions proved by \cite{[7]}.\\
{\bf Step.2} general case. We argue as in \cite{[13]}, then for $c\gg 1$, we set
$$
\begin{array}{lcl}
(dd^c u_{j})^{p}\wedge T-(dd^c u)^{p}\wedge T&=&[(dd^c u_{j})^{p}\wedge T-(dd^c\max(u,-c))^{p}\wedge T]\\& &+[(dd^c\max(u_j,-c))^{p}\wedge T-(dd^c\max(u,-c))^{p}\wedge T]\\& &+[(dd^c\max(u,-c))^{p}\wedge T-(dd^cu)^{p}\wedge T]\\&=&I_1+I_2+I_3.
\end{array}
$$
Observe that for every $c$, $\max(u_j,-c)$ converge in $cap_{m,T}$-capacity to $\max(u,-c)$. Therefore, by Step.1, it is clear that $I_2$ converges weakly to $0$, when $j\l +\infty$. To investigate the other terms let $\varepsilon>0$ and let $\varphi\in{\mathscr D}(\Omega)$, be a test function. Since $\max(u,-c)$ decreases to $u$ as $c\l +\infty$, we deduce from Thm.2 that for $c$ sufficiently large $\left|\langle I_3,\varphi\rangle\right|<\varepsilon$. In order to estimate $\langle I_1,\varphi\rangle$, we use lemma 3 to get ${\1}_{\{u_j>-c\}} (dd^c(\max(u_j,-c)))^p\wedge T = {\1}_{\{u_j>-c\}}(dd^c u_j)^p\wedge T$. Hence,
$$
\begin{array}{lcl}
\left|\langle I_1,\varphi\rangle\right|&=&
\left|\ds\int_{\{u_j\leq -c\}}\varphi[(dd^cu_j)^{p}\wedge T-(dd^c\max(u_j,-c))^{p}\wedge T]\right|\\&\leq& A\left(\ds\int_{\{u_j\leq -c\}}(dd^cu_j)^{p}\wedge T+\ds\int_{\{u_j\leq -c\}}(dd^c\max(u_j,-c))^{p}\wedge T\right).
\end{array}
$$
In view of lemma 2, we see that
$$
\begin{array}{lcl}
\ds\int_{\{u_j\leq -c\}}(dd^c\max(u_j,-c))^{p}\wedge T&\leq&\ds\int_{\{u_j\leq -c\}}\left(-1-\frac{2u_j}{c}\right)^p(dd^c\max(u_j,-c))^{p}\wedge T\\&\leq&
2^p\ds\int_{\{u_j<-c/2 \}}\left(-\frac{c}{2}-u_j\right)^p(dd^c\max(u_j/c ,-1))^{p}\wedge T
\\&\leq& 2^p(p!)\ds\int_{\{u_j<-c/2 \}}(dd^cu_j)^{p}\wedge T.
\end{array}
$$
It follows that $$|\langle I_1,\varphi\rangle|\leq A(1+2^p(p!))\ds\int_{\{u_j<-c/2 \}}(dd^cu_j)^{p}\wedge T.$$
Using the fact that $u_j\l u$ in $cap_{m,T}$ on each $E\Subset\Omega$ combined with hypothesis $(3)$, we easily see that $cap_{m,T}(\{u_j<-c/2 \})$ is uniformly convergent to zero for all $j$ as $c\l+\infty$. Hence, we conclude by the second hypothesis that there exists $c_0>0$ such that for any $c>c_0$, $|\langle I_1,\varphi\rangle|<\varepsilon$, for all $j$. Finally, we obtain the fact that $(dd^c u_{j})^{p}\wedge T-(dd^c u)^{p}\wedge T$ is  weakly convergent to zero, as $j\l +\infty$. In order to prove that $(dd^c u)^{p}\wedge T\ll cap_{m,T}$ on each subset $E\Subset\Omega$, we can proceed as in \cite{[13]}.
\end{proofof}
\section{Cegrell's classes associated to an $m$-positive closed current}
In this section, we associate to each $m$-positive closed current $T$, in the same way as \cite{[9]} and \cite{[10]}, three classes of Cegrell-type: ${\mathcal E}^{m,T}_0, {\mathcal F}^{m,T},{\mathcal E}^{m,T}$. In order to get important properties of such classes, it should be take more care than the case of the trivial current $T=1$ introduced by Cegrell \cite{[5]}, as well as the case $T=\beta^{n-m}$ introduced by Lu \cite{[10]}, since $T$ may have a large singular part. We investigate then the most important relative properties and we point out that some of them given in \cite{[9]} are not true. Next, we prove a Xing-type comparison principle inequality for the class ${\mathcal F}^{m,T}$ (see definition 5 below). The definitions and results involved in this section are quite obvious modification of those of the complex case $m=n$, or the strongly positive current $T=\beta^{n-m}$. For the proofs we shall following the line of \cite{[9]} and \cite{[10]}.
\subsection{Definitions and properties}
Let $\Omega$ be a bounded domain of $\mathbb{C}^n$. Let $T$ be an $m$-positive closed current of bidimension $(p,p)$ with $ p \geq 1$. Denote by ${\mathscr P}_m^-(\Omega)$ the cone of all negative $m$-sh functions on $\Omega$. We define
\begin{defn} \
\begin{enumerate}
\item $\mathcal{E}^{m,T}_0(\Omega)=\left\{u\in {\mathscr P}_m^-(\Omega)\cap L^\infty(\Omega),~{\ds\lim_{z\rightarrow\xi}} u(z) = 0~~ \forall \xi\in \partial\Omega\cap{\rm Supp}T, \int_{\Omega } (dd^cu)^p \wedge T < +\infty \right\}$
\item $\mathcal{F}^{m,T}(\Omega)= \left\{u \in{\mathscr P}_m^-(\Omega),~ \exists (u_{j})_j\subset {\mathcal E}^{m,T}_0(\Omega),\  u_{j}\downarrow u \ {\rm on\  \Omega,} \ {\sup_{j}}\int_{\Omega} (dd^c u_{j})^p \wedge T < +\infty\right\}$
\item Let $u\in {\mathscr P}_m^-(\Omega)$, we say that $u\in {\mathcal E}^{m,T}(\Omega)$ if for every $z_{0} \in \Omega $, there exist a neighborhood $\omega$ of $z_{0}$ in $\Omega$ and a decreasing sequence $(u_{j})_{j} \subset {\mathcal E}^{m,T}_0 (\Omega)$ such that $u_{j}\downarrow u$ on $\omega$
and $\sup_{j} \int_{\Omega} (dd^c u_{j})^q \wedge T < +\infty$.
\end{enumerate}
\end{defn}
In particular, if $m=n$, we obtain the well-known Cegrell's classes \cite{[5]}, for $T=1$ and the associated classes of \cite{[9]} for $T$ be a closed positive current. For any $m$ and $T =\beta^{n-m}$, we recover the classes introduced by Lu \cite{[10]}. Now, we start by listing some properties which can be proved by repeating the arguments in \cite{[5]} and \cite{[9]}.
\begin{propr} \
\begin{enumerate}
\item ${\mathcal E}^{m,T}_0(\Omega) \subset  {\mathcal F}^{m,T}(\Omega) \subset {\mathcal E}^{m,T}(\Omega)$.
\item If $\varphi \in {\mathcal E}^{m,T}_0(\Omega)$ and $\psi \in {\mathscr P}_m^-(\Omega)$, then $ \max(\varphi,\psi) \in {\mathcal E}^{m,T}_0(\Omega).$
\item The above classes in definition 5 are convex cones.
\item If $\varphi,\psi \in  {\mathcal E}^{m,T}_0(\Omega)$ then $\int_{\Omega}(-\psi)^{p+1}(dd^{c}\varphi)^p \wedge T \leq (p+1)! \sup(-\varphi)^{p} \int_{\Omega} (-\psi) (dd^c\psi)^p \wedge T.$
\item If $ u \in {\mathcal F}^{m,T}(\Omega)$ then $\int_{\Omega} (dd^c u)^p \wedge T < +\infty.$
\item Assume that $u, v\in{\mathscr P}_m^-(\Omega)$, and $\forall w\in\partial\Omega\cap{\rm Supp}T$, $\ds\lim_{\xi\rightarrow w}u(\xi)=\ds\lim_{\xi\rightarrow w}v(\xi)=0$. If $p=1$, then $\int_\Omega u dd^cv\wedge T=\int_\Omega v dd^cu\wedge T$.
\item If $u, v\in{\mathcal F}^{m,T}(\Omega)$, then
$\int_\Omega u dd^cv\wedge T=\int_\Omega v dd^cu\wedge T$.
\end{enumerate}
\end{propr}
The fourth property known as the Blocki's inequality, while the two later are "integration by parts" for the current $T$. Unfortunatly, some other properties of these classes given in \cite{[9]} for the case $m=n$, are not true. Namely, let us recall the following
\begin{rem}In their paper [9, Thm.5.3], the authors observed that "integration by parts" for positive closed current $T$ of dimension greater than one also holds true. We give here an example against which shows that the statement of Thm.5.3 (later proof) in \cite{[9]} is false. Let $ \Omega =B(0,1) $ in the unit ball $\mathbb{C}^3$ and for $z=(z_1,z_2,z_3)\in\Omega$, setting
$T=[z_{1}=0],\ u=\log|z|$ and $v=|z|^{2}-1,$ here $T$ is the current of integration on $\left\{z_1=0\right\}$. It is clear that the functions $u$ and $v$ satisfy the assumptions of Thm.5.3 of \cite{[9]}. Taking into account the fact that $\left(dd^c\log(|z_{2}|^2+|z_{3}|^2)^{1/2}\right)^2$ is nothing but the Dirac measure on the origin in $\bC^2$, it is not difficult to get $\int_{\Omega} u (dd^c v)^{2} \wedge T=\frac{-\pi^{2}}{32}\not=-1=\int_{\Omega} v(dd^cu)^{2}\wedge T$.
\end{rem}
\subsection{Approximation of $m$-sh functions and complex hessian operator for $\mathcal{E}^{m,T}$}
An essential tool in the study of complex pluripotential theory is the approximation of psh functions. Locally, this is a classic phenomenon by ordinary regularization. In \cite{[5]} Cegrell proved that such approximation is globally true for psh functions defined on an open hyperconvex set. Recently, Lu \cite{[10]} extends the result of Cegrell for $m$-sh functions defined on an $m$-hyperconvex open set. Recall that an open subset $\Omega$ of $\mathbb{C}^n$ is said $m$-hyperconvex if it is bounded, connected and there exists $\varphi \in{\mathscr P}_m^-(\Omega)$, continuous and exhaustive i.e, for every $c>0, ~
 \Omega_{c} =\{ z \in \Omega , \varphi(z) < -c\}  \Subset \Omega$. Our aim now is to establish the analogous approximation for the associated class ${\mathcal E}^{m,T}_0(\Omega)$.
\begin{theo}
Let $\Omega$ be an $m$-hyperconvex domain, $u\in{\mathscr P}_m^-(\Omega)$ such that $\ds\lim_{z\rightarrow w} u(z) = 0$, for each $w\in\partial \Omega\cap{\rm Supp}(T)$, where $T$ is an $m$-positive closed current of bidimension $(p,p)$, $p\geq 1$ such that $\int_{\Omega} (dd^c u)^p \wedge T < +\infty$. Then, there exists a decreasing sequence $(v_{j})_{j} \subset {\mathcal E}^{m,T}_0(\Omega)\cap{\cal C}(\ov\Omega)$ such that ${v_j}_{|\partial\Omega}= 0, \forall j \in \mathbb{N}$ and $\ds\lim_{j\rightarrow +\infty} v_{j}(z)= u(z), \forall z \in \Omega.$
\end{theo}
In the particular case $m=n$, Thm.4 was stated and proved by \cite{[9]}. We mention here that the proof of \cite{[9]} is incorrect. Indeed, the authors applied the comparison theorem of \cite{[7]} without the key assumption that $u$ is locally bounded. Subsequently, we give a complete proof of this result in the general case.
\begin{proof} Let $u_{k}= \max(u,-k)$. Applying Thm.3.1 in \cite{[10]} on each  $u_{k}$, there exists $(u^{j}_{k})_{j}$ a decreasing sequence ${\mathscr P}_m^-(\Omega)\cap {\cal C}(\overline{\Omega})$ with ${u^{j}_{k}}_{|\partial\Omega}\equiv 0 ~\forall j,k \in\mathbb{N}, \ds\lim_{j\rightarrow +\infty} u^{j}_{k}=u_{k}$. In view of the construction of $u_k^j$ (see the proof of Thm.3.1 in \cite{[10]}), observe that the sequence $(u^j_{k})_{k}$ is decreases also. It follows that
$$u\leq u_{j+1} \leq u^{j+1}_{j+1} \leq u^{j+1}_{j}\leq u^{j}_{j} \leq u^{j}_{s},\ \ \forall s \leq j.$$
Then by letting $j\rightarrow +\infty$ and $s\rightarrow +\infty$ in this order, we remark that the diagonal sequence $(v_{j}=u^j_{j})_{j}$ is decreasing pointwise to $u$. Finally, since $u_{k} \in{\mathscr P}_m^-(\Omega)\cap L^\infty(\Omega)$, then by using lemma 2 for $w_1=r-1$ and the Stokes formula, for all
 $\lambda>1$, we have
\vskip0.1cm
$\displaystyle\int_{\Omega} (dd^cv_{j})^p \wedge T =\displaystyle\int_{\{\lambda u_{j}<v_{j} \}}(dd^cv_{j})^p \wedge T \leq \lambda^{p} \displaystyle\int_{\Omega} (dd^c u_{j})^p \wedge T=\lambda^{p} \displaystyle\int_{\Omega} (dd^c u)^p \wedge T < +\infty.$
\end{proof}
According to the notations of \cite{[9]}, observe once again that the statement of Prop.5.16 in \cite{[9]} requires the key assumption: $\int_\Omega (dd^ch)^q\wedge T<+\infty$, because it was needed in the proof. Using the same argument, Thm.4 allows us to improve Prop.5.16 as follows:
\begin{pro} Let $u_1,\cdots,u_p\in{\mathcal F}^{m,T}(\Omega)$ and let $h\in{\mathscr P}_m^-(\Omega)$ such that $\ds\lim_{z\l\partial\Omega}h(z)=0$ and $\int_\Omega (dd^ch)^q\wedge T<+\infty$. If $g_j^q\in{\mathcal E}_0^{m,T}(\Omega)$ decreases to $u_q$ as $j\l +\infty$, $q=1,\cdots, p$ then
$$\ds\lim_{j\l +\infty}\int_\Omega h dd^cg^{1}_{j}\wedge\cdots \wedge dd^cg^{p}_{j}\wedge T=\int_\Omega h dd^cu^{1}\wedge \cdots \wedge dd^cu^{p}\wedge T .$$
\end{pro}
By repeating the same argument of \cite{[10]}, we shall extend the definition and the continuity of the complex hessian operator to the class $\mathcal{E}^{m,T}(\Omega)$. More precisely, we have
\begin{theo} Let $\Omega$ be an $m$-hyperconvex open subset of $\bC^n$, $T$ an $m$-positive closed current of bidimension $(p,p)$ on $\Omega$ and $u^{q} \in{\mathcal E}^{m,T}(\Omega), 1\leq q \leq p$. If $(g^{q}_{j})_{j} \subset{\mathcal E}^{m,T}_0(\Omega)$ is a decreasing sequence to $u^{q}$ when $j\l+\infty $ then $ dd^cg^{1}_{j}\wedge dd^cg^{2}_{j}\wedge\cdots \wedge dd^cg^{p}_{j}\wedge T$ converges weakly on $\Omega$ and the limit does not depend on the choice of the sequence $(g^{q}_{j})_{j}$.
\end{theo}
As a consequence, it is clear that the complex hessian operator is well defined for $u^q\in{\mathcal E}^{m,T}(\Omega)$. More precisely, if $u^{q}\in{\mathcal E}^{m,T}(\Omega), 1\leq q \leq p $, then $ dd^c u^{1} \wedge\cdots\wedge dd^c u^{p} \wedge T$ is a positive measure as the weak limit of the positive measure obtained in Thm.5.
\begin{defn} Consider the class ${\mathcal K}^{m,T}(={\mathcal K}^{m,T}(\Omega))\subset{\mathscr P}_m^-(\Omega)$, such that:
\begin{enumerate}
\item If $u\in{\mathcal K}^{m,T}$, $ v\in{\mathscr P}_m^-(\Omega)$ then $\max(u,v) \in{\mathcal K}^{m,T}.$
\item If $u\in{\mathcal K}^{m,T}$, $(u_{j})_j \subset{\mathscr P}_m^-(\Omega)\cap L^{\infty}_{loc}(\Omega), u_{j} \downarrow u$, then the sequence of measures $(dd^cu_{j})^{p}\wedge T$ weakly convergent.
\end{enumerate}
\end{defn}
\begin{pro} ${\mathcal E}^{m,T}(\Omega) \subset{\mathcal K}^{m,T}(\Omega).$
\end{pro}
\begin{rem} Notice that the above inclusion is in fact an equality for $m=n$ and $T=1$ by Cegrell \cite{[5]} as well as for $T=\beta^{n-m}$, by Lu \cite{[10]}. This leads to the following important fact: in both cited cases ${\mathcal E}^{m,T}(\Omega)$ is the largest class for which $(1)$ and $(2)$ of definition 6 holds true. As a consequence of the proof of this main result Cegrell has observed that ${\mathcal E}^{n,1}(\Omega)$ is locally in ${\mathcal F}^{n,1}(\Omega)$. This observation was extended by Lu \cite{[10]} for the case $T=\beta^{m-n}$. In their paper \cite{[9]} the authors have stated without proof the same result for ${\mathcal E}^{n,T}(\Omega)$. It is important here to mention that this is not yet clear. In fact, the arguments used firstly by Cegrell and later by Lu, build essentially on the existence of a relative extremal function, which is a key difficulty when $T$ is a positive closed current.
\end{rem}
\begin{proof} Let $u \in{\mathcal E}^{m,T}(\Omega)$ then by definition 5, for every $z_{0} \in \Omega$, there exist a neighborhood $\omega$ of $z_{0}$ in $\Omega$ and a decreasing sequence $(u_{j})_{j} \subset {\mathcal E}^{m,T}_0 (\Omega)$ such that $u_{j}\downarrow u$ on $\omega$
and $\sup_{j} \int_{\Omega} (dd^c u_{j})^q \wedge T < +\infty$. Let $v\in{\mathscr P}_m^-(\Omega)$ then according to Thm.3.1 in \cite{[10]}, there exists a decreasing sequence $(v_{j})_j\subset {\mathcal E}_0^m(\Omega)\cap {\cal C}(\overline{\Omega})$, such that $v_{j}\downarrow v$ on $\Omega$. Setting $\varphi_{j}=\max(u_{j}, v_{j})\in{\mathcal E}^{m,T}_0 (\Omega)$, it is clear that $\varphi_{j} \downarrow \max(u,v)=\varphi$ on $ \omega$. Moreover, by using lemma 2 for $w_1=r-1$, for every $\lambda > 1$,
$$\displaystyle\int_{\Omega} (dd^c \varphi_{j})^{p} \wedge T = \displaystyle\int_ {\{\lambda u_{j}< \varphi_{j}\}}(dd^c\varphi_{j})^{p} \wedge T \leq \lambda^{p}\displaystyle\int_{\Omega} (dd^c u_{j})^{p}\wedge T.$$
Hence, $\sup_{j}\int_{\Omega}(dd^c \varphi_{j})^p \wedge T \leq\sup_{j}\int_{\Omega} (dd^c u_{j})^p\wedge T <+\infty$ and therefore, $\varphi \in{\mathcal E}^{m,T}(\Omega)$. It follows that the first point in the definition of ${\mathcal K}^{m,T}$ was satisfied. Concerning the second one, let $(u_{j})_j\subset{\mathscr P}_m^-(\Omega) \cap L^{\infty}_{loc}(\Omega)$ such that $u_{j} \downarrow u\in{\mathcal E}^{m, T}(\Omega)$ and consider a function $\varphi \in {\mathcal E}^{m,T}_0 (\Omega)$. Then we have $ g_{j}= \max (u_{j}, m_{j} \varphi) \in{\mathcal E}^{m,T}_0 (\Omega),$ $g_{j}\downarrow u \in \mathcal{E}^{m,T}(\Omega)$, where $(m_{j})_j$ is any sequence decreasing to $-\infty$. Thank's to Thm.5, $(dd^cg_{j})^{p}\wedge T$ converges weakly to $(dd^cu)^{p}\wedge T$.
\end{proof}
\subsection{Xing-type comparison principle for ${\mathcal F}^{m,T}$}
Now, we prove the main result of this section. It's a Xing-type inequality for the class ${\mathcal F}^{m,T}(\Omega)$.
\begin{theo}
Let $\Omega$ be an $m$-hyperconvex domain of $\mathbb{C}^{n}$, $T$ an $m$-positive closed current of bidimension $(p,p)$, $1\leq p\leq n$, $u$ and $v\in \mathcal{F}^{m,T}(\Omega)$ such that $u\leq v$ in a neighborhood of ${\rm Supp}(T)$. Let $1\leq k\leq p$, $r\geq 1$ and $w_1\in r+\mathcal{E}_0^{m,T}(\Omega)$. Then, we have :
$$
\begin{array}{lcl}
\ds\frac{1}{k!}\displaystyle\int_{\Omega}(v - u)^p dd^cw_1\wedge\cdots \wedge dd^cw_p\wedge T &+&\displaystyle\int_{\Omega}(r-\omega_1)(dd^cv)^k\wedge dd^cw_{k+1}\wedge\cdots\wedge dd^cw_p\wedge T \\&\leq &  \displaystyle\int_{\Omega}(r-\omega_1)(dd^cu)^k\wedge dd^cw_{k+1}\wedge\cdots\wedge dd^cw_p\wedge T,
\end{array}
$$
for every $w_j\in{\mathcal E}^{m,T}(\Omega)$ such that $ 0\leq w_j\leq 1$, $j=2,...,k$ and $w_{k+1},\ldots, w_p\in \mathcal{F}^{m,T}(\Omega)$.
\end{theo}
Notice that in the case when $m=n$ and $T=1$, the hypothesis $w_j\in{\mathcal E}^{m,T}(\Omega)$ is superfluous since ${\cal {PSH}}\cap L_{loc}^\infty\subset {\mathcal E}^{n,1}$ (see \cite{[5]}). In that case we recover a result of \cite{[11]}. For the proof we need the following lemma with a proof similar to a result in \cite{[11]}.
\begin{lem} Let $\Omega$ be an open bounded subset of $\mathbb{C}^{n}$, $T$ an $m$-positive closed current of bidimension $(1,1)$ and $u, v\in{\mathscr P}_m(\Omega)\cap L^{\infty}(\Omega)$ such that $u\leq v$ in a neighborhood of ${\rm Supp}(T)$. Assume that for each $\xi \in\partial \Omega\cap{\rm Supp}(T),$ $\ds\lim_{z\rightarrow \xi}[u(z)-v(z)]=0$ then for every $w \in {\mathscr P}_m(\Omega), 0\leq w\leq 1$, we have
$$\displaystyle\int_\Omega (v-u)^k dd^cw\wedge T\leq k\displaystyle\int_\Omega (1-w)(v-u)^{k-1}dd^cu\wedge T.$$
\end{lem}
\begin{proofof}{\it Theorem 6.} \textbf{Step 1: case when $u,v\in\mathcal{E}_0^{m,T}(\Omega)$}.\\ For simplicity, setting $R=dd^cw_{k+1}\wedge\ldots\wedge dd^cw_p\wedge T$. We begin by the case $u=v$ in a neighborhood of $\partial\Omega\cap{\rm Supp}(T)$. By lemma 4 and the fact that $1- w_j\leq 1$, we get
$$
\begin{array}{lcl}
& &\displaystyle\int_\Omega (v-u)^k dd^cw_1\wedge\cdots\wedge dd^cw_p\wedge T=\\&=&\displaystyle\int_\Omega (v-u)^k dd^cw_1\wedge\cdots\wedge dd^cw_k\wedge R \\&\leq & k\displaystyle\int_\Omega(v-u)^{k-1}dd^cu\wedge dd^cw_1\wedge\cdots\wedge dd^cw_{k-1}\wedge R \\&\leq\cdots\leq& k!\displaystyle\int_\Omega (v-u)dd^cw_1\wedge(dd^cu)^{k-1}\wedge R \\&\leq & k!\displaystyle\int_\Omega (v-u)dd^cw_1\wedge\left[\sum_{s=0}^{k-1}(dd^cu)^s\wedge(dd^cv)^{k-s-1}\right]\wedge R \\&= & k!\displaystyle\int_\Omega (w_1-r)dd^c(v-u)\wedge\left[\sum_{s=0}^{k-1}(dd^cu)^s\wedge(dd^cv)^{k-s-1}\right]\wedge R  \\&= & k!\displaystyle\int_\Omega (r-w_1)dd^c(u-v)\wedge\left[\sum_{s=0}^{k-1}(dd^cu)^s\wedge(dd^cv)^{k-s-1}\right]\wedge R \\&= & k!\displaystyle\int_\Omega (r-w_1)\left[(dd^cu)^k-(dd^cv)^k\right]\wedge R.
\end{array}
$$
Consequently,
$$
\begin{array}{lcl}
\ds\frac{1}{k!}\displaystyle\int_{\Omega}(v - u)^k dd^cw_1\wedge\cdots \wedge dd^cw_p\wedge T &+&\displaystyle\int_{\Omega}(r-\omega_1)(dd^cv)^k\wedge dd^cw_{k+1}\wedge\cdots\wedge dd^cw_p\wedge T \\&\leq &  \displaystyle\int_{\Omega}(r-\omega_1)(dd^cu)^k\wedge dd^cw_{k+1}\wedge\cdots\wedge dd^cw_p\wedge T.
\end{array}
$$
Now, let us prove the general case. For $\varepsilon >0$, denote by $v_\varepsilon=\max(u, v-\varepsilon)$. It is clear that $v_\varepsilon\uparrow v$ on $\Omega$ when $\varepsilon\downarrow 0 ,  v_\varepsilon\geq u$ on $\Omega$ and $v_\varepsilon=u$ in a neighborhood of $\partial\Omega\cap{\rm Supp}(T)$. Using the preceding case, we have
\begin{equation}
\frac{1}{k!}\displaystyle\int_{\Omega}(v_\varepsilon - u)^k dd^cw_1\wedge\cdots \wedge dd^cw_p\wedge T +\displaystyle\int_{\Omega}(r-w_1)(dd^cv_\varepsilon)^k\wedge R \leq  \displaystyle\int_{\Omega}(r-w_1)(dd^cu)^k\wedge R.
\end{equation}
Now we claim that $(dd^cv_\varepsilon)^k\wedge R$ converges weakly to $(dd^cv)^k\wedge R$. In fact, since integration by parts is holds for $\mathcal{F}^{m,T}(\Omega)$, if we consider $h\in\mathcal{E}_0^{m,T}(\Omega)$, we have $$\int_\Omega h(dd^cv_\varepsilon)^k\wedge R=\int_\Omega v_\varepsilon dd^ch\wedge(dd^cv_\varepsilon)^{k-1}\wedge R\leq \int_\Omega vdd^ch\wedge(dd^cv_\varepsilon)^{k-1}\wedge R=\int_\Omega hdd^cv\wedge(dd^cv_\varepsilon)^{k-1}\wedge R.$$
Repeating the same argument, one get the inequality $\int_\Omega h(dd^cv_\varepsilon)^k\wedge R\leq \int_\Omega h(dd^cv)^k\wedge R$. To see the converse inequality, observe that
$$
\begin{array}{lcl}
\ds\int_\Omega h(dd^cv_\varepsilon)^k\wedge R&=&\ds\int_\Omega v_\varepsilon dd^ch\wedge(dd^cv_\varepsilon)^{k-1}\wedge R\\&\geq&\ds\int_\Omega (v-\varepsilon) dd^ch\wedge(dd^cv_\varepsilon)^{k-1}\wedge R\\&=&\ds\int_\Omega v dd^ch\wedge(dd^cv_\varepsilon)^{k-1}\wedge R-\varepsilon\ds\int_\Omega dd^ch\wedge(dd^cv_\varepsilon)^{k-1}\wedge R\\&=&\ds\int_\Omega h dd^cv\wedge(dd^cv_\varepsilon)^{k-1}\wedge R-\varepsilon\ds\int_\Omega dd^ch\wedge(dd^cu)^{k-1}\wedge R.
\end{array}
$$
In the later equality we use an integration by parts and the fact that $v_\varepsilon=u$ near the boundary. We continue in the same line, we easily obtain the following inequality :
$$\int_\Omega h(dd^cv_\varepsilon)^k\wedge R\geq\int_\Omega h(dd^cv)^k\wedge R-\varepsilon\sum_{s=0}^{k-1}\int_\Omega dd^ch\wedge(dd^cv)^s\wedge(dd^cu)^{k-1-s}\wedge R.$$
Since $\mathcal{F}^{m,T}(\Omega)$ is convex, we see that the integral $\int_\Omega dd^ch\wedge(dd^cv)^s\wedge(dd^cu)^{k-1-s}\wedge R$ is finite, and this complete the proof of the claim. On the other hand, since $0\leq v_\varepsilon -u \uparrow v-u$ when $\varepsilon \downarrow 0$ and $r-w_1$ is lower semi-continuous, by passing to the limit in $(4.1)$, when $\varepsilon\l 0$, we obtain:
$$\frac{1}{k!}\displaystyle\int_{\Omega}(v - u)^k dd^cw_1\wedge\cdots \wedge dd^cw_p\wedge T +\displaystyle\int_{\Omega}(r-\omega_1)(dd^cv)^k\wedge R \leq   \displaystyle\int_{\Omega}(r-w_1)(dd^cu)^k\wedge R.$$
\textbf{Step 2: case when $u,v\in\mathcal{F}^{m,T}(\Omega)$}.\\ By definition, there exists two sequences $(u_j)_j$ and $(v_j)_j $ in $\mathcal{E}^{m,T}_0(\Omega)$ such that $ u_j\downarrow u$ and $v_j\downarrow v$ on $\Omega$. Replacing $v_j$ by $\max(u_j,v_j)$, we may assume that $u_j\leq v_j ~~\forall j\geq 1$. Then, by applying Step 1, for every $1\leq j\leq s$, we have
$$
\begin{array}{lcl}
\ds\frac{1}{k!}\displaystyle\int_{\Omega}(v_j - u_s)^k dd^cw_1\wedge\cdots \wedge dd^cw_p\wedge T &+&\displaystyle\int_{\Omega}(r-\omega_1)(dd^cv_j)^k\wedge dd^cw_{k+1}\wedge\cdots\wedge dd^cw_p\wedge T \\&\leq &  \displaystyle\int_{\Omega}(r-\omega_1)(dd^cu_s)^k\wedge dd^cw_{k+1}\wedge\cdots\wedge dd^cw_p\wedge T.
\end{array}
$$
Since $w_1-r\in\mathcal{E}_0^{m,T}(\Omega)$, Prop.4 implies that
$$\ds\lim_{s\l +\infty}\int_{\Omega}(r-\omega_1)(dd^cu_s)^k\wedge dd^cw_{k+1}\wedge\cdots\wedge dd^cw_p\wedge T=\int_{\Omega}(r-\omega_1)(dd^cu)^k\wedge dd^cw_{k+1}\wedge\cdots\wedge dd^cw_p\wedge T.$$
Then similarly as in Step 1, by letting $s\rightarrow +\infty$ and $j\rightarrow +\infty$ in this order, we obtain the desired inequality:
$$
\begin{array}{lcl}
\ds\frac{1}{k!}\displaystyle\int_{\Omega}(v - u)^k dd^cw_1\wedge\cdots \wedge dd^cw_p\wedge T &+&\displaystyle\int_{\Omega}(r-\omega_1)(dd^cv)^k\wedge  dd^cw_{k+1}\wedge\cdots\wedge dd^cw_p\wedge T \\&\leq &  \displaystyle\int_{\Omega}(r-\omega_1)(dd^cu)^k\wedge dd^cw_{k+1}\wedge\cdots\wedge dd^cw_p\wedge T.\finpr
\end{array}
$$
\end{proofof}
Before ending this section let us state the following problem related to Thm.4:\\
{\bf Problem:} {\it  Is Theorem 4 remains true if we remove the assumption  $\int_{\Omega} (dd^c u)^p \wedge T < +\infty?$}

According to the papers \cite{[5]} and \cite{[10]}, if $T=\beta^{n-m}$ or more generally if $T$ satisfies the following hypothesis
  {\it there is an $m$-sh exhaustive and continuous function $v$ on $\overline\Omega$ such that Supp$\left(T\wedge(dd^cv)^p\right)\Subset\Omega$}, the answer to the previous question is positive.
\section{Complex hessian operator and $m$-potential current}
Let $T$ be a closed positive current of bidimension $(p,p)$. In this section, we associate to $T$, for each $m\geq p+1$, an $m$-potential by means of a local convolution of $T$ with $h_m\beta^{m-1}$, where $h_m$ is an $(n-m+1)$-sh function of Riesz-kernel-type. The important special case $m=n$ corresponds to the well-known Lelong-Skoda local potential of $T$ which had found a number of important applications in the study of the complex Monge-Amp\`ere operator (see \cite{[3]}). By following the work of \cite{[3]}, we prove firstly a result on the continuity of the complex hessian operator for a class of currents different to the one studied in the previously two sections.
\subsection{Complex hessian operator}
We have already seen that the complex hessian operator is well defined and continuous on decreasing sequence of locally bounded $m$-sh functions. It turns out in the following theorem that this happen also for a large class of $m$-positive not necessarily closed currents. Therefore, one can follow the lines of \cite{[3]} to prove :
\begin{theo}
 Let $1\leq p\leq m\leq n$, $S_m$ be a current of bidimension $(n-m+p,n-m+p)$ on $\Omega$ and $u_1,...,u_q$, $p\leq q$, are locally bounded $m$-sh functions on $\Omega$. Assume that there exists $(S_m^j)_j$ a sequence of smooth $(m-p,m-p)$-form on $\Omega$ such that $S_m^j$ is $m$-negative, $dd^cS_m^j$ is $m$-positive and $(S_m^j)_j$ decreases weakly to $S_m$. Then,
 \begin{enumerate}
 \item The sequence $S_m^j\wedge\beta^{n-m}\wedge dd^cu_1\wedge ...\wedge dd^cu_q$ converges weakly on $\Omega$ to a limit denoted by $S_m\wedge\beta^{n-m}\wedge dd^cu_1\wedge ...\wedge dd^cu_q$. This current is $m$-positive.
\item $\forall$ $E$, Borel subset of $\Omega$ and $\forall \varphi$ strongly positive continuous $(p-q,p-q)$-form, we have $$\ds\lim_{j\l +\infty}\int_E S_m^j\wedge\beta^{n-m}\wedge dd^cu_1\wedge ...\wedge dd^cu_q\wedge\varphi=\int_E S_m\wedge\beta^{n-m}\wedge dd^cu_1\wedge ...\wedge dd^cu_q\wedge\varphi.$$
\item $\forall L, K$ two compacts subset of $\Omega$, with $K\subset{\ds\mathop L^\circ}$, there exists a constant $c_{K,L}$ such that
$$\|S_m\wedge\beta^{n-m}\wedge dd^cu_1\wedge\cdots\wedge dd^cu_q\|_K\leq c_{K,L}\|S_m\|_L \|u_1\|_{\infty}(L)\cdots\|u_q\|_{\infty}(L),$$
where $\|S_m\|_L=-\int_LS_m\wedge\beta^{n-m+p}$ and $\|u\|_{\infty}(L)=\sup\{|u(z)|,\ z\in L\}$.
 \item Assume that $u_1^j,...,u_q^j$, are sequences of $m$-sh functions decreasing pointwise respectively to $u_1,...,u_q$, then
 $S_m\wedge\beta^{n-m}\wedge dd^cu_1^j\wedge ...\wedge dd^cu_q^j$ converges in the sense of currents to $S_m\wedge\beta^{n-m}\wedge dd^cu_1\wedge ...\wedge dd^cu_q$.
 \end{enumerate}
\end{theo}
Notice that Thm.7 generalizes a result of \cite{[3]} for the case $m=n$. Also, we point out that the assumptions about $S_m^j$, assert that the current $S_m\wedge\beta^{n-m}$ is $m$-negative.
\begin{proof} Let $(u_s^k)_k$ be the regularized sequence of $u_s$, $s=1,...,q$, then according to Prop.1 and lemma 1 the form $\beta^{n-m}\wedge dd^cu_1^k\wedge ...\wedge dd^cu_q^k$ is positive. Moreover, this form converges weakly to $\beta^{n-m}\wedge dd^cu_1\wedge ...\wedge dd^cu_q$. Let $\psi$ be a $(p-q,p-q)$-from strongly positive. Then, the form $S_m^j\wedge\beta^{n-m}\wedge dd^cu_1^k\wedge ...\wedge dd^cu_q^k\wedge\psi$ (which is negative because $S_m^j$ is $m$-negative) converges weakly to $S_m^j\wedge\beta^{n-m}\wedge dd^cu_1\wedge ...\wedge dd^cu_q\wedge\psi$, as $k\l +\infty$. It follows that, for every $j$, $S_m^j\wedge\beta^{n-m}\wedge dd^cu_1\wedge ...\wedge dd^cu_q\wedge\psi$ is negative. The rest of the proof of Thm.7 for the case $q<p$, is an easy adaptation of the arguments used by \cite{[3]}. Assume now that $p=q$. When $m=n$, it is a quite easy idea of Bedford-Taylor to pull-back the problem in $\Omega\times\bC$, and then using the case $q<p$ combined with the Fubini theorem. Unfortunatly, this is not the case when $m<n$, as shown remark 3. Therefore, we argue directly. Since the problem is local, by a techniques going back to \cite{[2]}, we can assume that $\Omega=\{\rho<0\}$, where $\rho$ is a psh smooth function in the neighborhood of $\overline\Omega$ and $-M\leq u_l\leq -1$, $l=1,...,p$ for a constant $M>1$. Let $\varphi$ be a positive test function. In order to prove statement $(1)$, it suffices to establish that $\int_\Omega \varphi S_m^j\wedge\beta^{n-m}\wedge dd^cu_1\wedge ...\wedge dd^cu_p$ is bounded from below, since it's a decreasing sequence. To this end let $K$, be a compact set containing the support of $\varphi$ and choose $\delta$ sufficiently small such that $K\subset\Omega_\delta=\{\rho<-\delta\}$. Let $v_l=\max\left(\frac{M}{\delta}\rho,u_l\right)$, $l=1,...,p$. It is clear that $v_l$ is $m$-sh and smooth near $\overline\Omega$, $v_l=u_l$ on $\Omega_\delta$ and $v_l=\frac{M}{\delta}\rho$ on the corona $\Omega\smallsetminus{\Omega_{\delta/M}}$. For simplicity, setting $\gamma=\beta^{n-m}\wedge dd^cu_1\wedge ...\wedge dd^cu_{p-1}$ and $\alpha=\beta^{n-m}\wedge dd^cv_1\wedge ...\wedge dd^cv_{p-1}$. Since $S_m^j$ is $m$-negative
$$\int_\Omega \varphi S_m^j\wedge\gamma\wedge dd^cu_p\geq c\int_K S_m^j\wedge\gamma\wedge dd^cu_p\geq c\int_{\Omega_\delta}S_m^j\wedge\alpha\wedge dd^cv_p\geq c\int_{\Omega_{\delta/{2M}}}S_m^j\wedge\alpha\wedge dd^cv_p.$$
Let $\psi\in{\mathscr D}\left(\Omega_{\delta/{3M}}\right)$, $\psi\geq 0$ and $\psi\equiv 1$ on $\overline\Omega_{\delta/{2M}}$. By regularization and passing to the limit, one can asumme that each $v_l$ is smooth. Then, by using the Stokes formula we get
\begin{equation}
\begin{array}{lcl}
\ds\int_{\Omega_{\delta/{2M}}}S_m^j\wedge\alpha\wedge dd^cv_p&\geq& \ds\int_{\Omega_{\delta/{3M}}}\psi S_m^j\wedge\alpha\wedge dd^c(v_p+M)\\&=&\ds\int_{\Omega_{\delta/{3M}}}(v_p+M)dd^c(\psi S_m^j)\wedge\alpha
\\&=&\ds\int_{\Omega_{\delta/{3M}}}(v_p+M)\psi dd^cS_m^j\wedge\alpha+\ds\int_{\Omega_{\delta/{3M}}}(v_p+M)dd^c\psi\wedge S_m^j\wedge\alpha\\& &-2\ds\int_{\Omega_{\delta/{3M}}}(v_p+M)dS_m^j\wedge d^c\psi\wedge\alpha\\&\geq&\ds\int_{\Omega_{\delta/{3M}}}(v_p+M)dd^c\psi\wedge S_m^j\wedge\alpha-2\ds\int_{\Omega_{\delta/{3M}}}(v_p+M)dS_m^j\wedge d^c\psi\wedge\alpha.
\end{array}
\end{equation}
The last inequality because $(v_p+M)\psi dd^cS_m^j\wedge\gamma$ is positive. Also, since $v_l=\frac{M}{\delta}\rho$ in the region where $dd^c\psi=0$, we see that the first integral in the same inequality is finite since it converges. Concerning the second integral, using the same argument, the Stokes formula yields
 $$\ds\int_{\Omega_{\delta/{3M}}}(v_p+M)dS_m^j\wedge d^c\psi\wedge\alpha=-\left(M/\delta\right)^{p}\ds\int_{\Omega_{\delta/{3M}}\smallsetminus\overline\Omega_{\delta/{2M}}}S_m^j\wedge\beta^{n-m}\wedge d((\rho+\delta)d^c\psi)\wedge(dd^c\rho)^{p-1}.$$
The later integral is obviously finite. To prove statement $(3)$ for the case $q=p$, we can use equation $(5.1)$ and the same statement for the case $q=p-1$.
\end{proof}
Arguing as in \cite{[3]}, we can establish the following corollary of Thm.7.
\begin{cor} With the same notation as in Theorem 7, we have
\begin{enumerate}
\item $S_m^j\wedge\beta^{n-m}\wedge dd^cu_1^j\wedge ...\wedge dd^cu_q^j$ converge weakly to $S_m\wedge\beta^{n-m}\wedge dd^cu_1\wedge ...\wedge dd^cu_q$.
\item $S_m\wedge\beta^{n-m}\wedge dd^cu_1\wedge ...\wedge dd^cu_q$ is not depending on the the sequence $(S_m^j)_j$ that satisfy the hypothesis of theorem 7.
\end{enumerate}
\end{cor}
As an example of $S_m$ in Thm.7, we can take $S_m=v\beta^{m-p}$, where $v$ is a negative $m$-sh function. Regarding definition 1, if $(v_j)_j$ is the standard regularization of $v$, one can easily check that $S_m^j=v_j\beta^{m-p}$ is $m$-negative and $dd^cS_m^j$ is $m$-positive. Building on the work of Ben Messaoud-El Mir \cite{[3]}, in the next subsection we give another example of such currents $S_m$, which involves some interesting properties.
\subsection{$m$-Potential current}
Let $T$ be a positive closed current of bidimension $(p,p)$ in an open set $\Omega$ of $\bC^n$. Let $\Omega_1\Subset\Omega$ and $\eta\in{\mathscr D}(\Omega)$, $0\leq \eta\leq 1 $ and $\eta\equiv 1 $ in a neighborhood of $\overline\Omega_1$.
\begin{defn} For every integer $p<m\leq n$, there exists a negative current of bidimension $(n-m+p+1,n-m+p+1)$ in $\bC^n$ denoted by $U_m=U_m(\Omega_1,T)$ and defined by:
  $$U_m(z) =\! -c_n\int_{x \in \bC^n}\!\eta (x) T(x)\wedge
{\beta^{m-1}(z-x)\over|z-x|^{\frac{2(m-1)}{n-m+1}}},\qquad{\rm where} \quad c_n= \ds {1\over (n-1)(4\pi )^n }.$$
\end{defn}
Let $h_m(x) =-c_n|x|^{\frac{-2(m-1)}{n-m+1}}=-c_n|x|^{-2\left( \frac{n}{n-m+1}-1\right)}$, be a Riesz kernel and put $\mu_m=\frac{c_n(m-1)}{n-m+1}$. A straightforward computation gives
\begin{equation}
dd^ch_m(x)=\ds\mu_m|x|^{\frac{-2n}{n-m+1}}\beta(x)-\frac{n\mu_m}{n-m+1}|x|^{-\frac{2(2n-m+1)}{n-m+1}}i\partial|x|^2\wedge\overline\partial|x|^2.
\end{equation}
 By using the Bin\^ome formula, it is not hard to get
$$\forall k\leq n-m+1,\qquad (dd^ch_m)^k\wedge\beta^{n-k}=(\mu_m)^k\frac{n(n-m-k+1)}{(n-m+1)}|x|^{-2kn/{n-m+1}}\beta^n .$$
It follows that $h_m$ is $(n-m+1)$-sh and satisfies $h_m\in L_{loc}^\alpha(\Omega)$, when $\alpha<{n^2-mn+n\over m-1}$ (see also \cite{[1]}). Denote by $K_m(x)= h_m(x)\beta^{m-1}(x)$, then $U_m(z) =(\eta. T)\star K_m (z)$. This means that the coefficients of $U_m$ are obtained as the convolution product of the coefficients of $\eta T$ (which are compactly supported measures) by $K_m$. It follows that the current $U_m$ has $L_{loc}^\alpha$, $\alpha<{{n^2-mn+n\over m-1}}$ as coefficients. It should be noted here that the border case $m=n$ is of key importance, since $h_n$ is the well-known Newton kernel and $U_n$ is nothing but the potential current of $T$ which is fundamental in the study of the Monge-Amp\`ere operator and the slicing theory of a positive closed current (see \cite{[3]}). Notice also that if $m=p+1$, $U_{p+1}$ is just a negative $(n-p)$-sh function. Let $\chi$ be a smooth positive compactly supported function in the unit ball of $\bC^n$ such that $\chi$ depending on $|z|$ and $\int_{\bC^n}\chi d\lambda_n=1$. Throughout the rest of this paper let $(\chi_j)_j$ be the associated regularization sequence to $\chi$. Denote by $K_m^j(z):=K_m\star\chi_j (z)= [h_m\star\chi_j(z)]\beta^{n-m}(z)$. In the remaining, we set
$$U_m^j(z) :=(\eta. T)\star K_m^j (z)=
   \int_{x \in \bC^n}\!\eta (x).(h_m\star\chi_j)(z-x) T(x)\wedge\beta^{m-1}(z-x) .$$
It is clear that $U_m^j$ is a smooth $(m-p-1,m-p-1)$-form in $\bC^n$. Since $h_m$ is $(n-m+1)$-sh, the regularization $(h_m\star\chi_j)_j$ is a sequence of smooth $(n-m+1)$-sh functions that decreases to $h_m$. It follows that the negative sequence $(U_m^j)_j$ decreases weakly to the current $U_m$. Let
$$U_m(z)=i^{(m-p-1)^2}\ds\sum_{|I|=|J|=m-p-1}U_{IJ}^mdz_I\wedge d\ov z_J .$$
A direct computations on the coefficients of the current $U_m$ gives
\begin{pro}
The trace $u_m(z)=\ds\sum_{I}U_{II}^m(z)$ is a negative $(n-m+1)$-sh function. Moreover, if $N=\min(p,m-p-1)$, for all $z\in\bC^n$, we have
$$u_m(z)=\left[\frac{2^{m-p-1}}{(n-m+p+1)!}\ds\sum_{s=0}^N\frac{(-1)^s(n-m)!s!(n-s)!}{(n-m-p-s)!(p-s)!}\right]\int_{x \in \bC^n}\eta (x) h_m(z-x)T(x)\wedge
\beta^p(x).$$
\end{pro}
\begin{proof} By rewriting the proof of \cite{[3]}, it suffices to prove the equality in the sense of currents. Observe that the form $\beta^{(m-1)}(z-x)$ acting only by its component $G$ of bidegree $(p,p)$ in $x$ and $(m-p-1,m-p-1)$ in $z$. Let $\gamma=\ds\sum_{j=1}^n 2i dx_j\w d{\overline z_j}$. Since $\beta(z-x)=\b(z)-\gamma
-{\overline \gamma }+\b(x)\,$ and these forms are even, by the Bin\^ome formula, we get
$$G=\ds\sum_{s=0}^N\ds{(m-1)!\over (m-1-p-s)!\,(p-s)!\,(s!)^2} \beta^{p-s}(x)\w (\gamma\w{\overline\gamma})^s\w \b^{m-1-p-s}(z) .$$
On the other hand, it is not hard to see that
$$(\gamma\w{\overline\gamma})^s=(s!)^2\ds\sum_{|I|=|J|=s}(-1)^s\,2^s\,i^{s^2}dz_J\w
d\overline z_I \w 2^s\,i^{s^2}\,dx_I\w d\overline x_J.$$
Therefore,
$$U_m(z)=\ds\sum_{s=0}^N\, \ds{(-1)^s(m-1)!\over (m-1-p-s)!\,(p-s)!} \left[\sum_{|I|=|J|=s}B_{I,J}^m(s,z)
\b^{m-1-p-s}(z)\w 2^s i^{s^2} dz_J\w d\overline z_I \right],$$
where, $$B_{I,J}^m(s,z)=\int_{x \in \bC^n}\!\eta (x)
h_m(z-x).T(x)\wedge\beta^{p-s}(x)\w 2^s i^{s^2} dx_I\w d\overline x_J .$$
Summing up the diagonal coefficients by a direct calculation of $U_m(z)\wedge\beta^{n-m+p+1}(z)$, it is not difficult to deduce the result.
\end{proof}
\begin{rem} Let $\alpha=\sum_{j,k=1}^nidx_j\wedge d\ov x_k$, then $\alpha$ is positive and $\alpha^s=s!\sum_{|I|=|J|=s} i^{s^2} dx_I\w d\overline x_J$. Using the notations of the preceding proof, it is clear that $\sum_{|I|=|J|=s}B_{I,J}^m(s,z)$ is a negative $(n-m+1)$-sh function, because it is the convolution product of the compactly supported positive measure $\eta (x)T(x)\wedge\beta^{p-s}(x)\wedge\alpha^s(x)$ with $h_m$. Similarly, each $B_{I,I}^m(s,z)$ is a negative $(n-m+1)$-sh function.
\end{rem}
Using a convolution argument by the smooth kernel $(\chi_j)_j$, the $m$-potential current $U_m$ shares the following properties, which extend the border case $m=n$, established by \cite{[3]}.
\begin{pro} Let $X$ be the interior of $\{\eta\equiv 1\}$. With the above notations we have
\begin{enumerate}
\item $(dd^ch_m)^{n-m+1}\wedge\beta^{m-1}=c_{n,m}\delta_0.\beta^n$, $c_{n,n}=1$ and $\delta_0$ is the Dirac measure at the origin.
\item $U_m^j=U_m\star\chi_j$.
\item There exists $R_m^j\in{\cal C}_{m-p,m-p}^\infty(\bC^n)$, such that $dd^cU_m^j=p_{2\star}[p_1^\star(\eta T) \wedge \tau^\star(dd^cK_m)]\star\chi_j+R_m^j.$ Moreover, $R_m^j$ converges in ${\cal C}_{m-p,m-p}^\infty(X)$, to a smooth $(m-p,m-p)$-form $R_m$ satisfying $$dd^cU_m=p_{2\star}[p_1^\star(\eta T) \wedge \tau^\star(dd^cK_m)]+R_m,$$ in the sense of currents, where the current $p_{2\star}[p_1^\star(\eta T) \wedge \tau^\star(dd^cK_m)]$ is positive, equals to $\eta T$ if $m=n$ and has $L_{loc}^\alpha$ as coefficients for $\alpha<n-m+1$, if $m<n$.
\end{enumerate}
\end{pro}
\begin{proof} Let $\alpha=dd^c\log|x|^2$ and $\gamma=i\partial|x|^2\wedge\overline{\partial}|x|^2$. By turning back to the equation after $(5.2)$, we see that
$(dd^ch_m)^{n-m+1}\wedge\beta^{m-1}$ is supported by the origin. On the other hand, using $(5.2)$ and the Bin\^ome formula, it is not hard to get
$$h_m(dd^ch_m)^{n-m+1}\wedge\beta^{m-1}=-c_n\mu_m^{n-m}|x|^{-2(n-1)}\beta^{n-1}+\frac{nc_n(n-m)\mu_m^{n-m}}{n-m+1}|x|^{-2n}\beta^{n-2}\wedge\gamma.$$
Since $\alpha^{n-1}=|x|^{-2(n-1)}\beta^{n-1}-(n-1)|x|^{-2n}\gamma\wedge\beta^{n-2}$, we deduce the equality
$$h_m(dd^ch_m)^{n-m+1}\wedge\beta^{m-1}=\frac{-c_n\mu_m^{n-m}(m-1)}{(n-1)(n-m+1)}|x|^{-2(n-1)}\beta^{n-1}-\frac{nc_n(n-m)\mu_m^{n-m}}{(n-m+1)(n-1)}\alpha^{n-1}.$$
It follows that
$$(dd^ch_m)^{n-m+1}\wedge\beta^{m-1}=dd^c(h_m(dd^ch_m)^{n-m+1}\wedge\beta^{m-1})=\frac{-c_n\mu_m^{n-m}(m-1)}{(n-1)(n-m+1)}\Delta(|x|^{-2(n-1)}).\beta^{n}.$$
It is well-known that $\Delta(-|x|^{-2(n-1)})=4n(n-1)I_{2n}\delta_0$, where $I_{2n}=\int_{\bR^{2n}}(|x|^2+1)^{-n-1}d\lambda(x)$. Then, we obtain the first statement
$$(dd^ch_m)^{n-m+1}\wedge\beta^{m-1}=4n(c_n)^{n-m+1}\left(\frac{m-1}{n-m+1}\right)^{n-m+1}I_{2n}\delta_0.\beta^n.$$
 To prove the second and the third statement we follows the arguments of \cite{[3]}. To check the third one, we denote by $p_1$ (resp.$p_2$) the first (resp.second) projection of $\bC^n\times\bC^n$ on $\bC^n$, i.e, $p_1(x,z)=x$ and $p_2(x,z)=z$. Let also $\tau$ be the function defined by $\tau(x,z)=z-x$. From the integral expressions of $U_m$ and $U_m^j$, we see that
$$U_m=p_{2\star}[p_1^\star(\eta T)\wedge\tau^\star(K_m)]\qquad {\rm and}\qquad U_m^j=p_{2\star}[p_1^\star(\eta T)\wedge\tau^\star(K_m^j)] .$$
This prove in particular the negativeness of $U_m$ and $U_m^j$. On the other hand, we have
\begin{equation}
\begin{array}{lcl}
dd^cU_m^j &=&p_{2\star}[p_1^\star(\eta T) \wedge \tau^\star(dd^cK_m^j)]+p_{2\star}[p_1^\star(dd^c(\eta T))\wedge \tau^\star(K_m^j)] \\&\ & +p_{2\star}[p_1^\star(d(\eta T)) \wedge\tau^\star(d^cK_m^j)]-p_{2\star}[p_1^\star(d^c(\eta T))\wedge\tau^\star(dK_m^j)]\\&=&p_{2\star}[p_1^\star(\eta T) \wedge \tau^\star(dd^cK_m^j)]+R_m^j.
\end{array}
\end{equation}
As the current $T$ is closed and the forms $ d\eta,d^c\eta$ and $dd^c\eta$ vanish on $X$, the sequence $R_m^j$ is smooth on $X$ and converges in $\C_{m-p,m-p}^{\infty}(X)$-topology to the smooth form
$$R_m=p_{2\star}[p_1^\star(dd^c(\eta T))\wedge \tau^\star(K_m)]+p_{2\star}[p_1^\star(d(\eta T)) \wedge\tau^\star(d^cK_m)]-p_{2\star}[p_1^\star(d^c(\eta T))\wedge\tau^\star(dK_m)].$$
If $m=n$, then by \cite{[3]}, $p_{2\star}[p_1^\star(\eta T) \wedge \tau^\star(dd^cK_n^j)]=(\eta T)\star\chi_j$. Assume now that $m<n$. For simplicity, setting
$$dd^ch_m=\varphi_m(x)\beta+\psi_m(x)\gamma,\ \ \varphi(x)=\ds\mu_m|x|^{\frac{-2n}{n-m+1}},\ \psi(x)=\frac{-n\mu_m}{n-m+1}|x|^{-\frac{2(2n-m+1)}{n-m+1} }.$$
It follows that
$$dd^cK_m^j=dd^ch_m^j\wedge\beta^{m-1}=\varphi_m^j(x)\beta^m+(\psi_m\gamma)\star\chi_j(x)\wedge\beta^{m-1}.$$ Consequently,
\begin{equation}
\begin{array}{lcl}
& &p_{2\star}[p_1^\star(\eta T) \wedge \tau^\star(dd^cK_m^j)]=\\&=&p_{2\star}[p_1^\star(\eta T) \wedge \tau^\star(\varphi_m^j(x)\beta^m)]+p_{2\star}\left[p_1^\star(\eta T) \wedge \tau^\star\left((\psi_m\gamma)\star\chi_j(x)\wedge\beta^{m-1}\right)\right]\\&=&P_m^j(x)+Q_m^j(x).
\end{array}
\end{equation}
Observe that the forms $\varphi_m(x)\beta$ and $\psi_m(x)\gamma$ have $L_{loc}^\alpha$ as coefficients for $\alpha<n-m+1$ (because $\gamma\leq |x|^2\beta$). Therefore, similarly as in the proof of the second statement, by the associativity of convolution action, we obtain
\begin{equation}
\begin{array}{lcl}
P_m^j(x)&=&\ds\int_{x \in \bC^n}\!\eta (x) T(x)\wedge\varphi_m^j(z-x)\beta^m(z-x)\\&=&\left(\ds\int_{x \in \bC^n}\!\eta (x) T(x)\wedge\varphi_m(z-x)\beta^m(z-x)\right)\star\chi_j\\&=&p_{2\star}[p_1^\star(\eta T) \wedge \tau^\star(\varphi_m(x)\beta^m(x))]\star\chi_j.
\end{array}
\end{equation}
And also,
\begin{equation}
\begin{array}{lcl}
Q_m^j(x)&=&\ds\int_{x \in \bC^n}\!\eta (x) T(x)\wedge(\psi_m\gamma)\star\chi_j(z-x)\wedge\beta^{m-1}(z-x)\\&=&\left(\ds\int_{x \in \bC^n}\!\eta (x) T(x)\wedge(\psi_m\gamma)(z-x)\wedge\beta^{m-1}(z-x)\right)\star\chi_j\\&=&p_{2\star}\left[p_1^\star(\eta T) \wedge \tau^\star\left(\psi_m(x)\gamma(x)\wedge\beta^{m-1}(x)\right)\right]\star\chi_j.
\end{array}
\end{equation}
On the other hand, we have
$$dd^cK_m=dd^ch_m\wedge\beta^{m-1}=\varphi_m(x)\beta^m+\psi_m(x)\gamma\wedge\beta^{m-1}.$$
Therefore, by summing up $(5.5)$ and $(5.6)$, the equation $(5.4)$ yields $$p_{2\star}[p_1^\star(\eta T) \wedge \tau^\star(dd^cK_m^j)]=\left(p_{2\star}[p_1^\star(\eta T) \wedge \tau^\star(dd^cK_m)]\right)\star\chi_j.$$ Finally, in
 virtue of $(5.3)$, we get
$$dd^cU_m^j=\left(p_{2\star}[p_1^\star(\eta T) \wedge \tau^\star(dd^cK_m)]\right)\star\chi_j+R_m^j.$$
In order to finish the proof, it remains to check that the current $p_{2\star}[p_1^\star(\eta T) \wedge \tau^\star(dd^cK_m)]$ is positive. To this end, it suffices to prove that $dd^cK_m$ is strongly positive. By using the well-known equality $\alpha^m=|x|^{-2m}\beta^m-m|x|^{-2m-2}\gamma\wedge\beta^{m-1}$, and by $(5.2)$, we have
\begin{equation}
\begin{array}{lcl}
dd^c K_m&=&\mu_m|x|^{\frac{-2n}{n-m+1}}\left[\beta^m-\frac{n}{n-m+1}|x|^{-2}\gamma\wedge\beta^{m-1}\right]\\&=&\mu_m|x|^{\frac{-2n}{n-m+1}}\left[|x|^{2m}\alpha^m+\left(\frac{m(n-m+1)-n}{n-m+1}\right)|x|^{-2}\gamma\wedge\beta^{m-1}\right]\\&=&\mu_m|x|^{\frac{2m(n-m+1)-2n}{n-m+1}}\alpha^m+\mu_m\left(\frac{m(n-m+1)-n}{n-m+1}\right)|x|^{-\frac{2(2n-m+1)}{n-m+1}}\gamma\wedge\beta^{m-1}.
\end{array}
\end{equation}
Since $\frac{m(n-m+1)-n}{n-m+1}=\frac{(m-1)(n-m)}{n-m+1}\geq 0$, the proof was completed.
 \end{proof}
 Using Thm.7 combined with the previous properties of $U_m$ and $U_m^j$, we can prove the following generalization of a result of \cite{[3]} for the case $m=n$.
 \begin{theo}
 Let $T,U_m,U_m^j$ as above such that $T$ is strongly positive. Assume that $u_1,...,u_q$ are locally bounded $m$-sh functions and $u_1^j,...,u_q^j$, $1\leq q\leq p+1$, are sequences of $m$-sh functions decreasing pointwise respectively to $u_1,...,u_q$. Then,
 \begin{enumerate}
 \item $U_m^j\wedge\beta^{n-m}\wedge dd^cu_1\wedge ...\wedge dd^cu_q$ converges in the sense of currents on $\Omega$ to a limit denoted by $U_m\wedge\beta^{n-m}\wedge dd^cu_1\wedge ...\wedge dd^cu_q$.
 \item $U_m\wedge\beta^{n-m}\wedge dd^cu_1^j\wedge ...\wedge dd^cu_q^j$ converges weakly on $\Omega$ to $U_m\wedge\beta^{n-m}\wedge dd^cu_1\wedge ...\wedge dd^cu_q$.
 \item $U_m^j\wedge\beta^{n-m}\wedge dd^cu_1^j\wedge...\wedge dd^cu_q^j$ converges weakly on $\Omega$ to $U_m\wedge\beta^{n-m}\wedge dd^cu_1\wedge ...\wedge dd^cu_q$.
 \item $dd^c(U_m\wedge\beta^{n-m}\wedge dd^cu_1\wedge ...\wedge dd^cu_q)=dd^cU_m\wedge\beta^{n-m}\wedge dd^cu_1\wedge
 ...\wedge dd^cu_q$ on $\Omega $.
 \end{enumerate}
\end{theo}
Notice that the technical assumption on $T$ to be strongly positive is especially related to the case $m<n$. In fact, when $m=n$, it suffices to assume that $T$ is positive (see \cite{[3]}).
\begin{proof} The problem is local, then without loss of generality we can assume that $\Omega$ is pseudoconvex i.e, $\Omega =\{\rho <0\}$, where $\rho$ is a smooth psh function in a neighborhood of $\overline\Omega$. Since $T$ is strongly positive, the currents $U_m, U_m^j$ are strongly negative. It follows from lemma 1, that $U_m^j$ is an $m$-negative form. On the other hand, by Prop.7, the form $R_m^j$ has a uniformly bounded coefficients on $\Omega$, thus, there exists a constant $A>0$ such that $R_m^j+A(dd^c\rho)^{m-p}$ is strongly positive. Hence, once again, Prop.7 imply that the form
$$dd^c\left(U_m^j+A\rho (dd^c\rho)^{m-p-1}\right)-p_{2\star}[p_1^\star(\eta T) \wedge \tau^\star(dd^cK_m)]\star\chi_j,$$
is $m$-positive. Since the form $p_{2\star}[p_1^\star(\eta T) \wedge \tau^\star(dd^cK_m)]\star\chi_j$ is $m$-positive, it follows that the form $dd^c(U_m^j+A\rho (dd^c\rho)^{m-p-1})$ is so. Therefore, by taking into account the fact that the form $A\rho (dd^c\rho)^{m-p-1}$ is $m$-negative, we are in measure to apply Thm.7 for the sequence $S_m^j=U_m^j+A\rho (dd^c\rho)^{m-p-1}$.
\end{proof}


\begin{thebibliography}
 {X-XX1}

\bibitem{[1]} \textbf{Abdullaev B. and Sadullaev A.}, Potential theory in the class of $m$-sh functions, {\sl Proc. Steklov Inst Math}, (2012) V279 pp.155-180.
\bibitem{[2]} \textbf{Bedford E. and Taylor B. A.}, A new capacity for Plurisubharmonic functions, {\sl Acta Math}, 149 (1982), 1-40.
\bibitem{[3]} \textbf{Ben Messaoud H et El Mir H.}, Op\'erateur de Monge-Amp\`ere et Tranchage des Courants Positifs Ferm\'es, {\sl J Geom Anal}, (2000), 10: 139-168.
\bibitem{[4]} \textbf{Blocki Z.}, Weak solutions to the complex hessian equation , {\sl Ann. Inst. Fourier, Grenoble}, 55, 5(2005), 1735-1756
 \bibitem{[5]} \textbf{Cegrell U.}, The general definition of the complex Monge-Amp\`ere operator, {\sl Ann. Inst. Fourier}, Grenoble, 54 (2004), 159-179.
\bibitem{[6]} \textbf{Coman D.}, Integration by parts for currents and applications to the relative capacity and Lelong numbers, {\sl Mathematica}, T 39(62), N 1, 1997, pp. 45-57.
 \bibitem{[7]} \textbf{Dabbek K. et Elkhadhra F.}, Capacit\'e associ\'ee \`a un courant positif ferm\'e, {\sl Documenta Math}, 11 (2006), 469-486.
 \bibitem{[8]} \textbf{Demailly J.-P.}, Complex analytic and differential geometry, available at: http://www-fourier.ujf-grenoble.fr, (1997).
\bibitem{[9]} \textbf{Le Mau H. and Nguyen D.}, Local $T$-pluripolarity and $T$-pluripolarity of a subset and some Cegrell's pluricomplex energy classes associated to a positive closed current,{\sl Viet.J.Math}, 37: 2-3(2009) 1-19.
\bibitem{[10]} \textbf{Lu H.}, A variational approach to complex hessian equations in $\bC^n$, {\sl Arxiv 1301: 6502v2 [math cv]} 07 Nov 2013.
\bibitem{[11]} \textbf{Nguyen K. and Pham. H}, A comparison principle for the complex Monge-Amp\`ere operator in Cegrell's classes et application , {\sl Trans. A.M.S}, 361(2009) 5539-5554.
\bibitem{[12]} \textbf{Ngoc N.}, Subsolution Theorem for the complex hessian equation, {\sl Univ. Lag. Acta. Math}, Fasciculus L(2012) 69-88.
\bibitem{[13]} \textbf{Xing Y.}, Complex Monge-Amp\`ere Measures of Plurisubharmonic Functions with Bounded Values Near the Boundary,  {\sl Canad. J.Math.}, 52(5), (2000), 1085-1100.


\end{thebibliography}
\end{document}